\documentclass{cocv}

\usepackage{caption2}
\usepackage{amssymb}
\usepackage{amsfonts}
\usepackage{mathrsfs}
\usepackage{empheq}
\usepackage{amsmath,amsthm}
\usepackage{graphicx}
\usepackage{color}
\usepackage{indentfirst}
\usepackage{hyperref}
\usepackage{float}
\usepackage{cases}
\usepackage{bm}

\usepackage{subfigure}
\usepackage{physics}
\usepackage{tikz}
\usepackage{mathdots}
\usepackage{yhmath}
\usepackage{cancel}
\usepackage{color}
\usepackage{siunitx}
\usepackage{array}
\usepackage{multirow}
\usepackage{amssymb}
\usepackage{gensymb}
\usepackage{tabularx}
\usepackage{extarrows}
\usepackage{booktabs}
\usetikzlibrary{fadings}
\usetikzlibrary{patterns}
\usetikzlibrary{shadows.blur}
\usetikzlibrary{shapes}

\usepackage[titletoc]{appendix}

\allowdisplaybreaks%

\newtheorem{thm}{Theorem}[section]

\newtheorem{lem}{Lemma}[section]
\newtheorem{prop}{Proposition}[section]
\theoremstyle{definition}
\newtheorem{defn}{Definition}[section]
\newtheorem{rem}{Remark}[section]
\numberwithin{equation}{section}

\def\i1n{i=1,\cdots,n}
\def\j1n{j=1,\cdots,n}
\def\ij1n{i,j=1,\cdots,n}

\def\R{\mathbb R}
\def\N{\mathbb N}

\def\C{\mathbb C}

\newcommand{\be}{\begin{equation}}
	\newcommand{\ee}{\end{equation}}
\newcommand{\beq}{\begin{equation*}}
	\newcommand{\eeq}{\end{equation*}}

\begin{document}
	\title{The Dichotomy Property in Stabilizability of $2 \times 2 $ Linear Hyperbolic Systems}
	\author{Xu Huang}\address{School of Mathematical Sciences, Fudan University, Shanghai 200433, China. E-mail: \texttt{xhuang20@fudan.edu.cn}.}
	\author{Zhiqiang Wang}\address{School of Mathematical Sciences and Shanghai Key Laboratory for Contemporary Applied Mathematics,
		Fudan University, Shanghai 200433, P. R. China. 
		E-mail: \texttt{wzq@fudan.edu.cn}. }
	\author{Shijie Zhou}\address{Research Institute of Intelligent Complex Systems,
		Fudan University, Shanghai 200433, China. E-mail: \texttt{sjzhou14@fudan.edu.cn}.}
	
	\date{April 27, 2023}
	\begin{abstract}
		This paper is devoted to discuss the stabilizability of a class of $ 2 \times2 $ non-homogeneous hyperbolic systems. Motivated by the example in \cite[Page 197]{CB2016}, we analyze the influence of the interval length $L$ on stabilizability of the system. By spectral analysis,  we prove that either the system is stabilizable for all $L>0$ or it possesses the dichotomy property: there exists a critical length $L_c>0$ such that the system is stabilizable for $L\in (0,L_c)$ but unstabilizable for $L\in [L_c,+\infty)$. In addition, for $L\in [L_c,+\infty)$, we obtain that the system can reach equilibrium state in finite time by backstepping control combined with observer. Finally, we also provide some numerical simulations to confirm our developed analytical criteria.
	\end{abstract}
	\maketitle

	\section{Introduction and Main Result}
	
	\par 
	
	Hyperbolic systems play a crucial role in representing physical phenomena and possess both theoretical and practical significance. Extensive research has focused on well-posedness and control problems, including the stability and stabilization of these systems. In particular, researchers have studied the exponential stability or stabilization of hyperbolic systems without source terms in both linear and nonlinear cases, under various boundary controls such as Proportional-Integral control and backstepping control \cite{Hu2019}. Previous works by \cite{BCT2015, CB2015, CBD2008, MR3567480, LRW2010, MR793047} have also investigated this topic. However, most physical equations, such as the Saint-Venant equations (see Chapter 5 of \cite{CORON2007}), Euler equations (see \cite{MR2740529} or \cite{MR3669664}), and Telegrapher equations, cannot neglect the source term. Therefore, it is crucial to investigate the dynamics of hyperbolic systems with source terms.
	\par Two main approaches have been used to achieve asymptotic stability of hyperbolic systems: analyzing the evolution of the solution along characteristic curves, as extensively studied in previous works such as \cite{MR2311519,MR1920513,LI1994,MR2974728,MR2411418}; and relying on a Lyapunov function approach, as thoroughly researched in \cite{CORON2007, CB2015, CBD2008, CBD2007, CBD2012, MR2805928, MR2871937, MR2513092, MR1925036}. Both of these approaches are concerned with obtaining stability for the system.
	\par Another strategy that has been studied is the Backstepping method, which aims to design a control law that achieves stabilization. The Backstepping method has been applied in \cite{Hu2021,Hu2019} and typically requires full-state feedback control. However, it is possible to achieve boundary state feedback through backstepping control by designing an appropriate observer, as demonstrated in \cite{An2017,An2018}.
	\par In \cite[Page 197]{CB2016}, Bastin and Coron mention that for some systems of balance laws, there is an intrinsic limit of stabilization under local boundary control. It is proved that the following system 
	\begin{align}\label{coron}
		\begin{cases}
			\partial_{t}y_{1}+\partial_{x}y_{1}+y_{2}=0,\quad &(t,x)\in(0,+\infty)\times(0,L),\\
			\partial_{t}y_{2}-\partial_{x}y_{2}+y_{1}=0,\quad &(t,x)\in(0,+\infty)\times(0,L),\\
			y_{2}(t,L)=y_{1}(t,L),\quad &t\in(0,+\infty),\\
			y_{1}(t,0)=ky_{2}(t,0),\quad &t\in(0,+\infty).
		\end{cases}
	\end{align}
	cannot be stabilized for any $k\in\R$ if $L\in(\frac{\pi}{c},+\infty)$. On the other hand, the system \eqref{coron} is stabilizable if $L\in(0,\frac{\pi}{2c})$ from \cite{CB2011}.
	However, there remains a gap on $L$ between the stabilizable and unstabilizable cases.
	\par In \cite{GM2019}, Gugat and Gerster analyze the limit of stabilizability for a network of hyperbolic systems. Remarkably, their results reveal that under certain conditions, the system may be inherently unstabilizable.
	\par These results inspire us to investigate whether the stabilizability of the system (\ref{coron}) possesses the dichotomy property on $L$. Here, the dichotomy property on $L$ can be described as follows: there exists a critical value $L_{c}>0$ such that:
	\begin{itemize}
		\item While $L\ge L_{c}$, the system is not stabilizable, i.e. the system cannot be exponentially stable for any discussed control.
		\item While $0<L<L_{c}$, the system is stabilizable, i.e. there exists certain control such that the system is exponentially stable.
	\end{itemize}	
	\par In this paper we discuss the boundary feedback stabilization of the following $ 2 \times2 $ linear hyperbolic systems over a bounded interval $[0,L]$:
	\begin{align}
		\begin{cases}\label{system}
			\partial_{t}y_{1}+\partial_{x}y_{1}+ay_{2}=0,\quad& (t,x)\in(0,+\infty)\times(0,L),\\
			\partial_{t}y_{2}-\lambda\partial_{x}y_{2}+by_{1}=0,\quad& (t,x)\in(0,+\infty)\times(0,L),\\
			y_{2}(t,L)=y_{1}(t,L),\quad& t\in(0,+\infty),\\
			y_{1}(t,0)=u(t),\quad& t\in(0,+\infty),\\
			\big(y_1(0,x),y_2(0,x)\big)=\left(y_1^0(x),y_2^0(x)\right),\quad& x\in(0,L).
		\end{cases}
	\end{align}
	where $\lambda>0$ and $a,b\in\R$ are given constants, $y_1^0, y_2^0 \in  L^2(0,L)$ are the initial data.  
	\par The boundary feedback law takes the proportional form
	\begin{align} 
		\label{boun}
		u(t)=ky_{2}(t,0),
	\end{align}
	where $k\in\mathbb{R}$ is the tuning parameter and $y_2(t,0)$ is the output measurement. 
	\par We are concerned about the exponential stability of the closed-loop system (\ref{system}). 
	\begin{defn}
		The linear hyperbolic system (\ref{system}) (\ref{boun}) is said to be $L^2$ exponentially stable if there exists $C>0$ and $\alpha>0$ such that, for every $(y_1^0(x),y_2^0(x))\in L^2(0,L)\times L^2(0,L)$
		the solution to the system (\ref{system}) (\ref{boun}) satisfies:
		\begin{equation}\label{L2}
			\Big\Vert\big(y_1(t,\cdot),y_2(t,\cdot)\big)\Big\Vert_{L^2(0,L)}\le C{\rm e}^{-\alpha t}\big\Vert (y_1^0,y_2^0)\big\Vert_{L^2(0,L)},\ t\ge 0.
		\end{equation}
	\end{defn}  
\par In this paper, we propose a method for finding the critical value $L_c$, given fixed parameters $a,b,\lambda$. Our approach is based on spectral analysis. For values of  $|k|\ge1$, we demonstrate that the closed-loop system (\ref{system}) (\ref{boun}) is not exponentially stable by identifying unstable eigenvalues, specifically those located on the right half of the complex plane. To achieve this, we approximate the characteristic function at infinity and use Rouché’s Theorem to obtain the roots. In the case of $|k|<1$, we introduce the function $N_{a,b,\lambda}(k,L)$ to represent the degree of the characteristic function, see \eqref{e13}, on the right side of the complex plane. As stated in Lemma \ref{lemma2}, $N_{a,b,\lambda}$ remains constant within each block that is separated by marginal curves determined by $\mathcal{A}_{a,b,\lambda}$. Furthermore, $N_{a,b,\lambda}\equiv 0$ within the block at the bottom. By applying Lemma \ref{lemma3}, we show that $N_{a,b,\lambda}$ increases by 1 when $(k,L)$ moves from one block to another block above it. Therefore, we conclude that the stability region is the block at the bottom. Finally, we determine the critical value $L_c$ using the analytical results obtained from the marginal curves determined by $A_{a,b,\lambda}$.  
	\begin{thm}\label{main theorem}
		Let $\lambda>0$ be fixed. Then, either the system (\ref{system}) (\ref{boun}) is stabilizable for all $L>0$ or it possesses the dichotomy property. More precisely, the expression of $L_c$ in terms of $a,b\in \R$ is explicitly given as follows (see Figure 1.1) : 
		\begin{equation}
			\begin{aligned}
				\displaystyle	
				L_c=\begin{cases}
					\sqrt{\frac{\lambda}{ab}}\pi,\quad &\text{if}\ a>0,b>0.\\
					\sqrt{\frac{\lambda}{ab}}\arccot(\frac{b-\lambda a}{2\sqrt{\lambda ab}}),\quad &\text{if}\ a<0, b<0.\\
					\sqrt{\frac{-\lambda}{ab}}\coth^{-1}(\frac{b-\lambda a}{2\sqrt{-\lambda ab}}),\quad &\text{if}\ -\lambda a>b>0.\\
					-\frac{2}{a},\quad &\text{if}\ b=0, a<0.\\
					+\infty,\quad &\text{if else.}
				\end{cases}
			\end{aligned}
		\end{equation}
		Here  $L_{c}:=+\infty$ means that the system is stabilizable for all $L>0$. 
	\end{thm}
	\begin{center}

		
		\tikzset{
			pattern size/.store in=\mcSize, 
			pattern size = 5pt,
			pattern thickness/.store in=\mcThickness, 
			pattern thickness = 0.3pt,
			pattern radius/.store in=\mcRadius, 
			pattern radius = 1pt}
		\makeatletter
		\pgfutil@ifundefined{pgf@pattern@name@_bcqvlgi96}{
			\pgfdeclarepatternformonly[\mcThickness,\mcSize]{_bcqvlgi96}
			{\pgfqpoint{0pt}{0pt}}
			{\pgfpoint{\mcSize}{\mcSize}}
			{\pgfpoint{\mcSize}{\mcSize}}
			{
				\pgfsetcolor{\tikz@pattern@color}
				\pgfsetlinewidth{\mcThickness}
				\pgfpathmoveto{\pgfqpoint{0pt}{\mcSize}}
				\pgfpathlineto{\pgfpoint{\mcSize+\mcThickness}{-\mcThickness}}
				\pgfpathmoveto{\pgfqpoint{0pt}{0pt}}
				\pgfpathlineto{\pgfpoint{\mcSize+\mcThickness}{\mcSize+\mcThickness}}
				\pgfusepath{stroke}
		}}
		\makeatother
		
		
		\tikzset{
			pattern size/.store in=\mcSize, 
			pattern size = 5pt,
			pattern thickness/.store in=\mcThickness, 
			pattern thickness = 0.3pt,
			pattern radius/.store in=\mcRadius, 
			pattern radius = 1pt}
		\makeatletter
		\pgfutil@ifundefined{pgf@pattern@name@_1fqgv8bx5}{
			\pgfdeclarepatternformonly[\mcThickness,\mcSize]{_1fqgv8bx5}
			{\pgfqpoint{0pt}{0pt}}
			{\pgfpoint{\mcSize}{\mcSize}}
			{\pgfpoint{\mcSize}{\mcSize}}
			{
				\pgfsetcolor{\tikz@pattern@color}
				\pgfsetlinewidth{\mcThickness}
				\pgfpathmoveto{\pgfqpoint{0pt}{\mcSize}}
				\pgfpathlineto{\pgfpoint{\mcSize+\mcThickness}{-\mcThickness}}
				\pgfpathmoveto{\pgfqpoint{0pt}{0pt}}
				\pgfpathlineto{\pgfpoint{\mcSize+\mcThickness}{\mcSize+\mcThickness}}
				\pgfusepath{stroke}
		}}
		\makeatother
		
		
		\tikzset{
			pattern size/.store in=\mcSize, 
			pattern size = 5pt,
			pattern thickness/.store in=\mcThickness, 
			pattern thickness = 0.3pt,
			pattern radius/.store in=\mcRadius, 
			pattern radius = 1pt}
		\makeatletter
		\pgfutil@ifundefined{pgf@pattern@name@_lqcoc3w0z}{
			\pgfdeclarepatternformonly[\mcThickness,\mcSize]{_lqcoc3w0z}
			{\pgfqpoint{0pt}{0pt}}
			{\pgfpoint{\mcSize}{\mcSize}}
			{\pgfpoint{\mcSize}{\mcSize}}
			{
				\pgfsetcolor{\tikz@pattern@color}
				\pgfsetlinewidth{\mcThickness}
				\pgfpathmoveto{\pgfqpoint{0pt}{\mcSize}}
				\pgfpathlineto{\pgfpoint{\mcSize+\mcThickness}{-\mcThickness}}
				\pgfpathmoveto{\pgfqpoint{0pt}{0pt}}
				\pgfpathlineto{\pgfpoint{\mcSize+\mcThickness}{\mcSize+\mcThickness}}
				\pgfusepath{stroke}
		}}
		\makeatother
		
		
		\tikzset{
			pattern size/.store in=\mcSize, 
			pattern size = 5pt,
			pattern thickness/.store in=\mcThickness, 
			pattern thickness = 0.3pt,
			pattern radius/.store in=\mcRadius, 
			pattern radius = 1pt}
		\makeatletter
		\pgfutil@ifundefined{pgf@pattern@name@_vhll9ywue}{
			\pgfdeclarepatternformonly[\mcThickness,\mcSize]{_vhll9ywue}
			{\pgfqpoint{0pt}{0pt}}
			{\pgfpoint{\mcSize}{\mcSize}}
			{\pgfpoint{\mcSize}{\mcSize}}
			{
				\pgfsetcolor{\tikz@pattern@color}
				\pgfsetlinewidth{\mcThickness}
				\pgfpathmoveto{\pgfqpoint{0pt}{\mcSize}}
				\pgfpathlineto{\pgfpoint{\mcSize+\mcThickness}{-\mcThickness}}
				\pgfpathmoveto{\pgfqpoint{0pt}{0pt}}
				\pgfpathlineto{\pgfpoint{\mcSize+\mcThickness}{\mcSize+\mcThickness}}
				\pgfusepath{stroke}
		}}
		\makeatother
		
		
		\tikzset{
			pattern size/.store in=\mcSize, 
			pattern size = 5pt,
			pattern thickness/.store in=\mcThickness, 
			pattern thickness = 0.3pt,
			pattern radius/.store in=\mcRadius, 
			pattern radius = 1pt}
		\makeatletter
		\pgfutil@ifundefined{pgf@pattern@name@_vyn3uszso}{
			\pgfdeclarepatternformonly[\mcThickness,\mcSize]{_vyn3uszso}
			{\pgfqpoint{0pt}{0pt}}
			{\pgfpoint{\mcSize}{\mcSize}}
			{\pgfpoint{\mcSize}{\mcSize}}
			{
				\pgfsetcolor{\tikz@pattern@color}
				\pgfsetlinewidth{\mcThickness}
				\pgfpathmoveto{\pgfqpoint{0pt}{\mcSize}}
				\pgfpathlineto{\pgfpoint{\mcSize+\mcThickness}{-\mcThickness}}
				\pgfpathmoveto{\pgfqpoint{0pt}{0pt}}
				\pgfpathlineto{\pgfpoint{\mcSize+\mcThickness}{\mcSize+\mcThickness}}
				\pgfusepath{stroke}
		}}
		\makeatother
		\tikzset{every picture/.style={line width=0.75pt}} 
		
		\begin{tikzpicture}[x=0.6pt,y=0.6pt,yscale=-1,xscale=1]
			
			\draw  (87,218) -- (575,218)(366,51) -- (366,381) (568,213) -- (575,218) -- (568,223) (361,58) -- (366,51) -- (371,58)  ;
			\draw  [draw opacity=0][pattern=_bcqvlgi96,pattern size=18.75pt,pattern thickness=0.75pt,pattern radius=0pt, pattern color={rgb, 255:red, 221; green, 160; blue, 221}] (365,57) -- (566,57) -- (566,216) -- (365,216) -- cycle ;
			\draw  [draw opacity=0][pattern=_1fqgv8bx5,pattern size=18.75pt,pattern thickness=0.75pt,pattern radius=0pt, pattern color={rgb, 255:red, 184; green, 233; blue, 134}] (366.3,219.46) -- (90.55,58.17) -- (367.68,60.59) -- cycle ;
			\draw  [draw opacity=0][pattern=_lqcoc3w0z,pattern size=18.75pt,pattern thickness=0.75pt,pattern radius=0pt, pattern color={rgb, 255:red, 184; green, 233; blue, 134}] (365,216) -- (565,216) -- (565,371) -- (365,371) -- cycle ;
			\draw [color={rgb, 255:red, 208; green, 2; blue, 27 }  ,draw opacity=1 ][line width=1.5]    (90.55,58.17) -- (363.61,215.87) ;
			\draw  [draw opacity=0][pattern=_vhll9ywue,pattern size=18.75pt,pattern thickness=0.75pt,pattern radius=0pt, pattern color={rgb, 255:red, 250; green, 244; blue, 174}] (91,58.17) -- (366.3,219) -- (91,219) -- cycle ;
			\draw  [draw opacity=0][pattern=_vyn3uszso,pattern size=18.75pt,pattern thickness=0.75pt,pattern radius=0pt, pattern color={rgb, 255:red, 141; green, 238; blue, 238}] (91,219) -- (365,219) -- (365,378) -- (91,378) -- cycle ;
			
			\draw (374,38) node [anchor=north west][inner sep=0.75pt]  [font=\Large]  {$b$};
			\draw (561,230) node [anchor=north west][inner sep=0.75pt]  [font=\Large]  {$a$};
			\draw (420,95) node [anchor=north west][inner sep=0.75pt]    {$L_{c} =\sqrt{\frac{\lambda }{ab}} \pi $};
			\draw (256,95) node [anchor=north west][inner sep=0.75pt]  [font=\large]  {$L_{c} =+\infty $};
			\draw (432,290) node [anchor=north west][inner sep=0.75pt]  [font=\large]  {$L_{c} =+\infty $};
			\draw (89,33) node [anchor=north west][inner sep=0.75pt]    {$b=-\lambda a$};
			\draw (100,153) node [anchor=north west][inner sep=0.75pt]  [font=\scriptsize]  {$L_{c} =\sqrt{\frac{\lambda }{ab}}\coth^{-1}\frac{b-\lambda a}{2\sqrt{-\lambda ab}}$};
			\draw (145,280) node [anchor=north west][inner sep=0.75pt]  [font=\scriptsize]  {$L_{c} =\sqrt{\frac{\lambda }{ab}}\arccot\frac{b-\lambda a}{2\sqrt{\lambda ab}}$};
			
		\end{tikzpicture}
		~\\	Figure 1.1: The expression of $L_c$
	\end{center}
	\begin{rem}
	By setting $\lambda=1$ and $a=b=c>0$ in Theorem \ref{main theorem}, we obtain $L_c=\frac{\pi}{c}$. This implies that the system \eqref{coron} is stabilizable for $L\in(0,\frac{\pi}{c})$, but not stabilizable for $L\in [\frac{\pi}{c},+\infty)$. The result presented in this paper bridges the gap between the stabilizable region $(0,\frac{\pi}{2c})$ (established in \cite{CB2011}) and the unstabilizable region $[\frac{\pi}{c},+\infty)$ (demonstrated in \cite{CB2016}) for the system (\ref{coron}).
	\end{rem}
	\begin{rem}
		While $L$ is sufficiently small, we can establish a Lyapunov function to demonstrate that the system \eqref{system} \eqref{boun} is exponentially stable. For instance, if $ab>0$, $|k|<\varepsilon<1$, we define
		$$
		V(y_1,y_2)\triangleq\int_{0}^{L}\bigg(\frac{y_{1}^{2}(t,x)}{\eta(x)}+\frac{\eta(x)y_{2}^{2}(t,x)}{\lambda}\bigg)dx,
		$$
	where
		 $\eta(x)$ satisfies $\eta'(x)=(1+\varepsilon)(|a|+\frac{|b|}{\lambda}\eta^{2})$, $\eta_{\varepsilon}(0)=\varepsilon$, i.e.
		$$
		\begin{aligned}
			\eta(x)=\sqrt{\frac{\lambda a}{b}}\tan\big{(}\sqrt{\frac{ab}{\lambda}}(1+\varepsilon)x+\arctan\sqrt{\frac{b}{\lambda a}}\varepsilon\big{)}.
		\end{aligned}
		$$
		Therefore, if $L<L_{k}=
		\frac{1}{1+\varepsilon}\sqrt{\frac{\lambda}{ab}}\big{(}\arctan\sqrt{\frac{b}{\lambda a}}-\arctan\sqrt{\frac{b}{\lambda a}}\varepsilon\big{)}$, we verify that $V(y_1,y_2)$ is a Lyapunov function and the corresponding system is exponentially stable.
	\end{rem}
	\begin{rem}
		The general hyperbolic system with the rightward speed $\lambda_{1}>0$  and leftward speed $\lambda_{2}>0$
		\begin{equation}
			\label{general system}
			\begin{cases}
				\partial_{t}y_{1}+\lambda_{1}\partial_{x}y_{1}+ay_{2}=0,\quad& (t,x)\in(0,+\infty)\times(0,L),\\
				\partial_{t}y_{2}-\lambda_{2}\partial_{x}y_{2}+by_{1}=0,\quad& (t,x)\in(0,+\infty)\times(0,L),\\
				y_{2}(t,L)=y_{1}(t,L),\quad& t\in(0,+\infty),\\
				y_{1}(t,0)=ky_2(t,0),\quad& t\in(0,+\infty).
			\end{cases}
		\end{equation}
		can be reduced, through a scaling of the space variable $x \rightarrow \lambda_1 x$, to a system with rightward speed $1$ and leftward speed $\lambda_1>0$ in the form of (\ref{system}). Thus, Theorem \ref{main theorem} can be extended to the general system (\ref{general system}).
	\end{rem}
	
	According to Theorem \ref{main theorem}, the proportional feedback control (\ref{boun}) cannot stabilize the system (\ref{system}) for $L>L_c$. Therefore, an alternative control approach is worth exploring in this case. Building upon the works of Hu et al. \cite{Hu2021,Hu2019} and Holta et al. \cite{An2017}, we develop a Backstepping control combined with observer design that stabilizes the system even when $L\geq L_c$, and without the need to observe the full state. Notably, the proposed control law drives the system to its zero equilibrium in finite time. More details are presented in Section 5.

	The main contribution of this paper can be summarized in three aspects: 
	1) we provide a complete characterization of the stabilizability of the hyperbolic system \eqref{system} under proportional feedback control \eqref{boun} for all cases;
	2) we show that the stabilizability of the system exhibits a dichotomy property on the interval $L$, indicating a clear boundary between the stabilizable and non-stabilizable regions;
	3) we propose a new control method that combines backstepping control with observer design to stabilize the system when the proportional control fails.

	\par The organization of this paper is as follows. In Section 2, we provide some preliminaries including Spectral Mapping Property and Implicit Function Theorem, which will be used in the following Sections. In Section 3, we provide the proof of Theorem \ref{main theorem}.  In Section 4, we provide some numerical simulations to confirm our developed analytical criteria in Section 3. In Section 5, We give a sketch of the construction of the Backstepping control with observer design for the case of $L\geq L_c$.
	
	\textbf{Notations}. In this paper, we use standard notation and terminology in complex analysis and algebra. Specifically, $\mathbb{C}+$, $\mathbb{C}-$, and $\mathbb{C}_0$ denote the sets of complex numbers with positive real parts, negative real parts, and zero real parts, respectively. We use $\overline{\mathbb{C}+}$ to denote the set $\mathbb{C}_+\cup\mathbb{C}_0$. The sets of integers, positive integers, and non-negative integers are denoted by $\mathbb{Z}$, $\mathbb{N}^{*}$, and $\mathbb{N}$, respectively. The imaginary unit is denoted by $\rm i$ such that $\rm i=\sqrt{-1}$. For $\sigma\in\mathbb{C}$, we use ${\rm Re}\sigma$, ${\rm Im}\sigma$, ${\rm arg}\sigma$, and $|\sigma|$ to denote the real part, imaginary part, principal value of argument, and norm of $\sigma$, respectively. For an analytic function $f$, an open subset $\Omega\subseteq \mathbb{C}$ and $b\in\mathbb{C}$, ${\rm deg}(f,\Omega,b)$ denotes the number of roots of the equation $f(z)=b$, counted by multiplicty.

	\section{Preliminaries}
	Applying the results in Lichtner \cite{Lich2008}, we have the following lemmas:
	\begin{lem}\label{lem1}
		Let $S(t) (t \ge 0) $ be the $C^{0}$ semigroup on $L^{2}(0, L)$ that corresponds to the solution map of (\ref{system}) (\ref{boun}), and let $\mathscr{A}$ be the infinitesimal generator of the
		semigroup $S(t) (t \ge 0)$. Let us denote by $\sigma_{p}(\mathscr{A})$ and $\sigma(\mathscr{A})$ the point spectrum and the
		spectrum of $\mathscr{A}$, respectively. Then,
		$\sigma(\mathscr{A})=\sigma_{p}(\mathscr{A})$. Moreover, $\mathscr{A}$ has the Spectral Mapping Property (SMP), that is
		$$
		\text{(SMP)}\quad \sigma({\rm e}^{\mathscr{A}t})\backslash\{0\}=\overline{{\rm e}^{\sigma(\mathscr{A})t}}\backslash\{0\},\quad \text{for}\quad t\ge0.
		$$
		Hence (SMP) contains spectrum determined growth
		$$\omega(\mathscr{A})=s(\mathscr{A}),$$
		with  $\omega(\mathscr{A}):=\inf\{\omega\in\mathbb{R}\,|\,\exists\, M=M(\omega):\Vert S(t)\Vert\le M{\rm e}^{\omega t}, \forall\ t\ge 0\}$, $s(\mathscr{A}):=\sup\{\Re(\mu)\,|\,\mu\in\sigma(\mathscr{A})\}$
	\end{lem}
	
	From Lemma \ref{lem1}, we obtain the following proposition:
	\begin{prop} \label{prop1}
		The system (\ref{system}) (\ref{boun}) is not exponentially stable if and only if $s(\mathscr{A})\ge0$.
	\end{prop}
	
	\par  We will apply the analytic implicit function theorem in the proof of Lemma \ref{lemma3}. The Implicit Function Theorem from \cite{KH2002} is stated following:
	\begin{lem}\label{l1}
		Let $\mathcal{B}\subset C^n\times C^m$ be an open set, $f=(f_1,...,f_m):\mathcal{B}\to C^m$ a holomorphic mapping, and $(z_0,w_0)\in\mathcal{B}$ a point with $f(z_0,w_0)=0$ and
		$$
		\det(\frac{\partial f_\mu}{\partial z_\mu}(z_0,w_0)\Bigg|\begin{aligned}
			&\mu=1,...,m\\&\nu=n+1,...,n+m
		\end{aligned})\ne0.
		$$
		Then there is an open neighborhood $U=U'\times U''\subset \mathcal{B}$ and a holomorphic map $g:U'\to U''$ such that
		$$\{(z,w)\in U'\times U'': f(z,w)=0\}=\{(z,g(z)):z\in U'\}.
		$$
	\end{lem}
	
	\section{Proof of Theorem \ref{main theorem}}
	In this section, we give the proof of Theorem \ref{main theorem}.
	We first establish the characteristic equation.
\par	Let $\sigma\in\mathbb{C}$ be the eigenvalue of the system, we look for a nontrivial solution $(y_{1},y_{2})^{\rm T}$ of the system with the form:
	\begin{equation}
		\begin{cases}
			y_{1}(t,x)={\rm e}^{\sigma t}f(x),\quad (t,x)\in[0,+\infty)\times(0,L), \\
			y_{2}(t,x)={\rm e}^{\sigma t}g(x),\quad (t,x)\in[0,+\infty)\times(0,L).
		\end{cases}
	\end{equation}
\par	$f(x),g(x)$ are the corresponding eigenfunctions of the $y_1(t,x),y_2(t,x)$.
	Such a $(y_{1},y_{2})^{\rm T}$ is a solution of the system if and only if
	\begin{eqnarray}
		\sigma f+\partial_ xf+ag=0, \label{e3}\\
		\sigma g-\lambda \partial_ xg+bf=0, \label{e4}\\
		f(L)=g(L),\label{e5e}\\ 
		f(0)=kg(0). \label{e5}
	\end{eqnarray}
\par	From Eq.\eqref{e3} and Eq.\eqref{e4}, we have:
	$$\lambda \partial_{xx}f+(\lambda-1)\sigma\partial_{x}f+(ab-\sigma^2)f=0.$$
\par	We divide into two classes  $\sigma^2\ne\frac{4\lambda ab}{(\lambda+1)^2}$ and $\sigma^2=\frac{4\lambda ab}{(\lambda+1)^2}$.
\par	For the first case, $\sigma^2\ne\frac{4\lambda ab}{(\lambda+1)^2}$, thus all eigenvalues are single. We obtain
	\begin{equation}
		f={\rm e}^{\xi x}(A{\rm e}^{\eta x}+B{\rm e}^{-\eta x}),
	\end{equation}
	\begin{equation}
		g=-\frac{1}{a}{\rm e}^{\xi x}\big{(}A(\sigma+\xi+\eta){\rm e}^{\eta x}+B(\sigma+\xi-\eta){\rm e}^{-\eta x}\big{)},
	\end{equation}
	where 
	$$
	\xi=\frac{-(\lambda-1)}{2\lambda}\sigma,\ 
	\eta^{2}=\frac{(\lambda+1)^{2}\sigma^{2}-4\lambda ab}{4\lambda^{2}}\ne 0,\ A,B\in\mathbb{C}.
	$$
\par	Eq.\eqref{e5e} and Eq.\eqref{e5} yield
	\begin{eqnarray}
		A\big{(}a+k(\sigma+\xi+\eta)\big{)}+B\big{(}a+k(\sigma+\xi-\eta)\big{)}=0, \label{e7}\\
		A\big{(}a+\sigma+\xi+\eta\big{)}{\rm e}^{\eta L}+B\big{(}a+\sigma+\xi-\eta\big{)}{\rm e}^{-\eta L}=0. \label{e8}
	\end{eqnarray}
\par	By computing the characteristic determinant of Eq.\eqref{e7} and Eq.\eqref{e8}, we obtain that there exists $(A,B)\in\mathbb{C}^2\backslash\{(0,0)\}$ such that Eq.\eqref{e7} and Eq.\eqref{e8} hold if and only if	
	\begin{equation}
		\label{e11}
		(k-1)\cosh(\eta L)-\big{(}(k+1)\frac{\lambda+1}{2\lambda}\sigma+(\frac{kb}{\lambda}+a)\big{)}\frac{\sinh(\eta L)}{\eta}=0.
	\end{equation}
\par	For the second case, $\sigma^2=\frac{4\lambda ab}{(\lambda+1)^2}$. In this case, we have
	\begin{equation}
		f={\rm e}^{\xi x}(Ax+B),
	\end{equation}
	\begin{equation}
		g=-\frac{1}{a}{\rm e}^{\xi x}\big{(}A(\xi+\sigma)x+(A+\xi B+\sigma B)\big{)}.
	\end{equation}
\par	Then Eq.\eqref{e5e} and Eq.\eqref{e5} yield
	\begin{eqnarray}
		A\big{(}aL+\sigma L+\xi L\big{)}+B\big{(}a+\sigma+\xi\big{)}=0, \label{e71}\\
		Ak+B\big{(}a+k\sigma+k\xi\big{)}=0. \label{e81}
	\end{eqnarray}
\par	By computing the characteristic determinant of Eq.\eqref{e71}, Eq.\eqref{e81}, we obtain that there exists $(A,B)\in\mathbb{C}^2\backslash\{(0,0)\}$ such that Eq.\eqref{e71} Eq.\eqref{e81} hold if and only if	
	\begin{equation}
		\label{e11-}
		(k-1)-(k\frac{\lambda+1}{2\lambda}\sigma+a\big{)}\big{(}1+\frac{(\lambda+1)\sigma}{2\lambda a}\big{)}L=0,
	\end{equation}
which can be included in Eq. \eqref{e11-} for $\eta=0$ if we define $ \frac{\sinh(\eta L)}{\eta}\triangleq L$ for $\eta=0$.	
\par	Note that $\cosh(\eta L)$ and $\dfrac{\sinh(\eta L)}{\eta}$ are analytical functions with respect to $\eta^2$, which further implies that the left side of Eq. \eqref{e11} is analytic with respect to $\sigma$. Denote by \begin{equation}\label{e13}
		F_{a,b,\lambda,k,L}(\sigma)\triangleq (k-1)\cosh(\eta L)-\big{[}(k+1)\frac{\lambda+1}{2\lambda}\sigma+(\frac{kb}{\lambda}+a)\big{]}\frac{\sinh(\eta L)}{\eta},
	\end{equation}
	with
	$$\eta^{2}=\frac{(\lambda+1)^{2}\sigma^{2}-4\lambda ab}{4\lambda^{2}},$$
	and
	$$
	\mathcal{U}_F\triangleq\Big\{\sigma\Big|F_{a,b,\lambda,k,L}(\sigma)=0 \Big\}.
	$$
\par	The set of the roots of $F_{a,b,\lambda,k,L}(\sigma)$ satisfy
	$$\mathcal{U}_F=\sigma({\mathscr{A}})=\sigma_p({\mathscr{A}}).$$
\par	By Lemma \ref{lem1}, we obtain that a sufficient and necessary condition for the stability of the system \eqref{system} \eqref{boun}: 
	\begin{equation}
		\mathcal{U}_F\cap\C_+\cap\C_0=\emptyset.
	\end{equation}
\par	Denote by 
	$$N_{a,b,\lambda}(k,L)\triangleq {\rm deg}(F_{a,b,\lambda,k,L}(\sigma),\mathbb{C}_+,0).$$
\par	Our goal is to establish the region for parameter $a,b,\lambda,k,L$ such that  $N_{a,b,\lambda}(k,L)=0$. To acheive this, we fix $a,b,\lambda$ and figure out the stability region for $k,L$ on $\mathbb{R}\times [0,+\infty)$. We divide into two case $|k|\ge 1$ and $|k|<1$.
	\subsection{$|k|\ge 1$}
	\begin{lem} \label{|k|>1}
		For every $a,b,\lambda\in\R$ and  $L>0$, if $|k|\ge 1$, the system \eqref{system}\eqref{boun} is non-exponentially stable.
	\end{lem}
	\begin{proof}[{\bf Proof of Lemma \ref{|k|>1}}]
		Obviously, by Proposition \ref{prop1} we just need to prove	
		\begin{center}
			$\forall L\in\mathbb{R}, |k|\ge 1$, the system (\ref{system}) (\ref{boun}) has $s(\mathscr{A})\ge0$.
		\end{center}
\par		We analyze the solution for the characteristic equation (\ref{e11}) in various cases:
		~\\ \textbf{Case\quad $k=1$}
	\par While $k=1$, Eq.\eqref{e11} yields the following equation with  $\eta^{2}=\frac{(\lambda+1)^{2}\sigma^{2}-4\lambda ab}{4\lambda^{2}}:$
		\begin{equation}\label{ch1}
			\big{(}(\lambda+1)\sigma+(b+\lambda a)\big{)}({\rm e}^{2\eta L}-1)=0.
		\end{equation} 
\par The solutions of Eq.\eqref{ch1} include $$\sigma_{n}=\frac{2\lambda}{(\lambda+1)}\sqrt{\frac{ab}{\lambda}-\frac{n^{2}\pi^{2}}{L^{2}}}\quad(n\in\mathbb{N}),$$
		which are located on the imagniary axis for sufficiently large $n\in \mathbb{N}^*$.
		
		\noindent\textbf{Case\quad $k=-1$}
	\par While $k=-1$, Eq.\eqref{e11} yields the following equation with  $\eta^{2}=\frac{(\lambda+1)^{2}\sigma^{2}-4\lambda ab}{4\lambda^{2}}:$
		$$-2\eta({\rm e}^{2\eta L}+1)=(a-\frac{b}{\lambda})({\rm e}^{2\eta L}-1).$$ 
	\par	Considering the solutions on the imaginary axis $\eta=\alpha {\rm i}(\alpha\ne0,\alpha\in\mathbb{R})$, we obtain 
		$$\cot \alpha L=\frac{a-\frac{b}{\lambda}}{2}\alpha,$$ which has infinitely solutions. This implies that  $F_{a,b,\lambda,k,L}(\sigma)$ has infinitely many roots on the imaginary axis.
		
		~\\ \textbf{Case\quad $|k|>1$}
		
		Multiplying $2{\rm e}^{\eta L}$ by two sides of Eq.\eqref{e11}, we obtain:
		$$
		H(\sigma)\triangleq(k-1)(1+{\rm e}^{-2Q(\sigma)L})-\bigg[(k+1)\frac{\lambda+1}{2\lambda}\frac{\sigma}{Q(\sigma)}+\frac{kb+\lambda a}{Q(\sigma)}\bigg](1-{\rm e}^{-2Q(\sigma)L})=0,
		$$
		where $Q(\sigma)$ denotes a single-valued branch that satisfied  $4\lambda^2Q(\sigma)^2=(\lambda+1)^{2}\sigma^{2}-4\lambda ab$, and
		$$
		\lim\limits_{|\sigma|\rightarrow+\infty}\frac{Q(\sigma)}{\sigma}=\frac{\lambda+1}{2\lambda}.
		$$
		\par This enlightens us to estimate $H(\sigma)$ as
		$$
		G(\sigma)\triangleq(k-1)(1+{\rm e}^{-\frac{(\lambda+1)\sigma}{\lambda}L})-(k+1)(1-{\rm e}^{-\frac{(\lambda+1)\sigma}{\lambda}L}),
		$$
		with $|\sigma|\rightarrow+\infty$, which vanishes at infinitely many points in $\C_+$ reading
		$$
		\begin{aligned}\
			\sigma_{k,n}=\frac{\lambda}{(\lambda+1)L}(\ln|k|+2\hat{n}\pi i),\ \hat{n}=
			\begin{cases}
				n,\quad k\in(1,+\infty),\\
				n+\frac{1}{2},\quad k\in(-\infty,-1).
			\end{cases}
		\end{aligned}
		$$
		\par Now we estimate $G(\sigma)-H(\sigma)$ as $|\sigma|\rightarrow+\infty, \sigma\in\C_+$.
		$$
		\begin{aligned}
			&G(\sigma)-H(\sigma)\\
			=&2k\big({\rm e}^{-\frac{(\lambda+1)\sigma}{\lambda}L}-{\rm e}^{-2Q(\sigma)L}\big)+\frac{kb+\lambda a}{Q(\sigma)}(1-{\rm e}^{-2Q(\sigma)L})+(k+1)\bigg(\frac{\lambda+1}{2\lambda}\frac{\sigma}{Q(\sigma)}-1\bigg)(1-{\rm e}^{-2Q(\sigma)L}).
		\end{aligned}
		$$
		\par Notice that:
		$$
		\lim\limits_{|\sigma|\rightarrow+\infty, \sigma\in\C_+}\Big|2Q(\sigma)-\frac{(\lambda+1)\sigma}{\lambda}\Big|L=\Big|\frac{4ab}{2\lambda Q(\sigma)+(\lambda+1)\sigma}\Big|L=0.	
		$$
	\par	We obtain
		$$
		\limsup\limits_{|\sigma|\rightarrow+\infty, \sigma\in\C_+}\big|{\rm e}^{-\frac{(\lambda+1)\sigma}{\lambda}L}-{\rm e}^{-2Q(\sigma)L}\big|=\big|{\rm e}^{-\frac{(\lambda+1)\sigma}{\lambda}L}\big|\cdot\big|1-{\rm e}^{\big(-2Q(\sigma)+\frac{(\lambda+1)\sigma}{\lambda}\big)L}\big|=0.
		$$
	\par	Furthermore,
		$$
		|{\rm e}^{-2Q(\sigma)L}|={\rm e}^{-2{\rm Re}(Q(\sigma))L}=({\rm e}^{-2{\rm Re}(\sigma)L})^{\frac{{\rm Re}(Q(\sigma))}{{\rm Re}(\sigma)}},
		$$
		which implies that
		$$
		\limsup\limits_{|\sigma|\rightarrow+\infty, \sigma\in\C_+}|{\rm e}^{-2Q(\sigma)L}|\le 1.
		$$
	\par	Then we have
		\begin{equation}\label{e12}
			\lim\limits_{|\sigma|\rightarrow+\infty, \sigma\in\C_+}|G(\sigma)-H(\sigma)|=0.
		\end{equation}
\par		We apply Rouch\'{e}'s Theorem for holmotopic functions to show that $H(\sigma)$ has infinitely roots in $\mathbb{C}_+$.
\par		At $\sigma_{k,n}$, let $\varepsilon>0$ small enough such that
		$$\Omega^{\varepsilon}_{k,n}\triangleq\big\{\sigma\in\mathbb{C}\big||G(\sigma)|<\varepsilon,|\sigma-\sigma_{k,n}|<1\big\}\subset\mathbb{C}_+.$$
\par		Notice that $\Omega^{\varepsilon}_{k,n}$ is a bounded open set of $\mathbb{C}_+$ and
		$\deg(G(\sigma),\Omega^{\varepsilon}_{k,n},0)=1$.
\par If $\varepsilon>0$ is small enough, we have 
		$$
		\partial\Omega^{\varepsilon}_{k,n}\subset\big\{\sigma\in\mathbb{C}\big||G(\sigma)|=\varepsilon\big\},\ \ 
		\Omega^{\varepsilon}_{k,n}\subset\mathbb{C}_+,\quad\forall n\in\mathbb{Z}.
		$$
\par	For all $\sigma\in\partial\Omega^{\varepsilon}_{k,n}$, from Eq.\eqref{e12}, there exists $N>0$, such that for all $n>N$, we have
		$$
		|G(\sigma)-H(\sigma)|<\varepsilon=|G(\sigma)|,\quad \forall \sigma\in\partial\Omega_{k,n}^{\varepsilon}.
		$$
\par		Applying Rouch\'{e}'s Theorem, we obtain
		$$\deg(H(\sigma),\Omega^{\varepsilon}_{k,n},0)=\deg(G(\sigma),\Omega^{\varepsilon}_{k,n},0)=1,\quad \forall n>N. $$
\par		Thus, there exists $N$ such that while $n>N$, $F(\sigma)$ vanishes at $\hat{\sigma}_{k,n}\in\Omega^{\varepsilon}_{k,n}\subset\mathbb{C}_{+}$. For $|k|>1$, the spectrum of the solution $s(\mathscr{A})$ satisfies $s(\mathscr{A})>0$.
\par These analysis for three cases complete the proof of Lemma \ref{|k|>1}.
	\end{proof}
	\subsection{$|k|<1$}
	From Lemma \ref{|k|>1}, only we choose $|k|<1$ can the system be exponentially stable. We then prove the main theorem by spectral analysis.
	
	Denote by
	$$
	\mathcal{A}_{a,b,\lambda}\triangleq\big\{(k,L)||k|<1,L\ge 0,\ \text{there exists}\ \beta\in\R\ \text{such that}\ F_{a,b,\lambda,k,L}(\mathrm{i}\beta)=0\big\}.
	$$
	\begin{lem}\label{lemma2}
		For any continuous path $(k(t),L(t))_{t\in[0,1]}\subset(-1,1)\times[0,+\infty)$, if
		$$\big\{(k(t),L(t))|t\in[0,1]\big\}\cap\mathcal{A}_{a,b,\lambda}=\emptyset,$$ 
		then $N_{a,b,\lambda}(k(t),L(t))$ is a constant on $[0,1]$.
	\end{lem}
	\begin{proof}[{\bf Proof of Lemma \ref{lemma2}}]Since $k(t)$ is continuous with respect to $t$, there exists $\delta>0$ such that $|k(t)|<1-\delta$.
		First, we prove that for $L\geq 0$ and $|k|<1-\delta$, there exists $R>0$, such that 
		\begin{equation}\label{e9}
			(\mathcal{U}_F\cap\C_+)\subset\{\sigma||\sigma|\le R\}.
		\end{equation} 
\par	Recall the roots of $F_{a,b,\lambda,k,L}(\sigma)$ satisfy
		\begin{equation}\label{chara-1}
			(k-1)(1+{\rm e}^{-2Q(\sigma) L})-\bigg{[}(k+1)\frac{\lambda+1}{2\lambda}\frac{\sigma}{Q(\sigma)}+\frac{(\frac{kb}{\lambda}+a)}{Q(\sigma)}\bigg{]}(1-{\rm e}^{-2Q(\sigma) L})=0.
		\end{equation}
\par	Here $Q(\sigma)$ is the single-valued branch that satisfied  $4\lambda^2Q(\sigma)^2=(\lambda+1)^{2}\sigma^{2}-4\lambda ab$, and 
		\begin{equation}\label{e1}
			\lim\limits_{|\sigma|\rightarrow+\infty}\frac{\sigma}{Q(\sigma)}=\frac{2\lambda}{\lambda+1},\quad\lim\limits_{|\sigma|\rightarrow+\infty}|Q(\sigma)|=+\infty.
		\end{equation}
\par		Eq.\eqref{chara-1} is equivilent to
		\begin{equation}\label{chara-2}
			(k-1)(1+{\rm e}^{-2Q(\sigma) L})-(k+1)\frac{\lambda+1}{2\lambda}\frac{\sigma}{Q(\sigma)}(1-{\rm e}^{-2Q(\sigma) L})=\frac{(\frac{kb}{\lambda}+a)}{Q(\sigma)}(1-{\rm e}^{-2Q(\sigma) L}).
		\end{equation}
\par		Note that 
		$$
		|{\rm e}^{-2Q(\sigma)L}|={\rm e}^{-2{{\rm Re}Q(\sigma)}L}
		=({\rm e}^{-2{\rm Re}\sigma L})^\frac{{\rm Re}Q(\sigma)}{{\rm Re}\sigma},$$
		which implies that
		$$
		\limsup_{|\sigma|\to+\infty,\sigma\in\cup\overline{\mathbb{C}_+}} |{\rm e}^{-2Q(\sigma)L}|\leq 1.$$
\par	As $|\sigma|\to+\infty,\sigma\in\overline{\mathbb{C}_+}$,  the right side of Eq.\eqref{chara-2} can be estimated as
		\begin{equation}\label{e2}
			\lim\limits_{|\sigma|\to+\infty,\sigma\in\cup\overline{\mathbb{C}_+}}\frac{(\frac{kb}{\lambda}+a)}{Q(\sigma)}(1-{\rm e}^{-2Q(\sigma) L})=0.
		\end{equation}
\par	The left side of Eq.\eqref{chara-2} can be estimated as
		\begin{equation}\label{e6}
			\begin{aligned}
				&\liminf_{|\sigma|\to+\infty,\sigma\in\cup\overline{\mathbb{C}_+}}\bigg|(k-1)(1+{\rm e}^{-2Q(\sigma) L})-(k+1)\frac{\lambda+1}{2\lambda}\frac{\sigma}{Q(\sigma)}(1-{\rm e}^{-2Q(\sigma) L})\bigg|\\
				\geq&\liminf_{|\sigma|\to+\infty,\sigma\in\cup\overline{\mathbb{C}_+}}\Big|(k-1)-(k+1)\frac{\lambda+1}{2\lambda}\frac{\sigma}{Q(\sigma)}\Big|-\Big|(k-1)+(k+1)\frac{\lambda+1}{2\lambda}\frac{\sigma}{Q(\sigma)}\Big|\cdot\Big|{\rm e}^{-2Q(\sigma) L}\Big|\\
				\ge&2-2|k|\geq\delta.
			\end{aligned}
		\end{equation}
\par		The limit in Eq.\eqref{e2} and Eq.\eqref{e6} are taken uniformly with respect to $L\geq 0, |k|<1-\delta$. Therefore, there exists a sufficently large $R>0$ such that Eq.\eqref{chara-1} has no roots located in $\sigma\in\overline{\mathbb{C}_+}\cap\{\sigma||\sigma|\geq R\}.$ Then we establish Eq.\eqref{e9}.
		Eq.\eqref{e9} yields
		$$
		N_{a,b,\lambda}(k(t),L(t))=\deg(F_{a,b,\lambda,k(t),L(t)}(\sigma),\C_+,0)=\deg(F_{a,b,\lambda,k(t),L(t)}(\sigma),\C_+\cap\{\sigma||\sigma|<R\},0).
		$$
\par	 We denote the contour $C\triangleq\{R{\rm e}^{{\rm i}\theta}|\theta:-\frac{\pi}{2}\to\frac{\pi}{2}\}\cup\{\mathrm{i}\beta| \beta:R\to-R\}$. If $\{k(t), L(t)\}\cap\mathcal{A}_{a,b,\lambda}=\emptyset$, using argument principle, we obtain
		\begin{equation}\label{arg}
			N_{a,b,\lambda}(k(t),L(t))=\frac{1}{2\pi\mathrm{i}}\int_{C}\frac{F'_{a,b,\lambda,k(t),L(t)}(\sigma)}{F_{a,b,\lambda,k(t),L(t)}(\sigma)}{\rm d}\sigma.
		\end{equation}
\par		Since $F_{a,b,\lambda,k(t),L(t)}(\sigma)\ne 0$ on $C$, the right side of Eq.\eqref{arg} is continuous with respect to $t$. Furthermore, since $N_{a,b,\lambda}(k(t),L(t))$ is an integer, we know that $N_{a,b,\lambda}(k(t),L(t))$ is a constant due to the discontiuity between two different integers.
	\end{proof}
	As shown in Fig. \ref{figAA}, $\mathcal{A}_{a,b,\lambda}$ divides the $k-L$ plane into several blocks.
	\begin{figure}[H]
		\centering
		\subfigure[]{
			\includegraphics[width=0.45\textwidth]{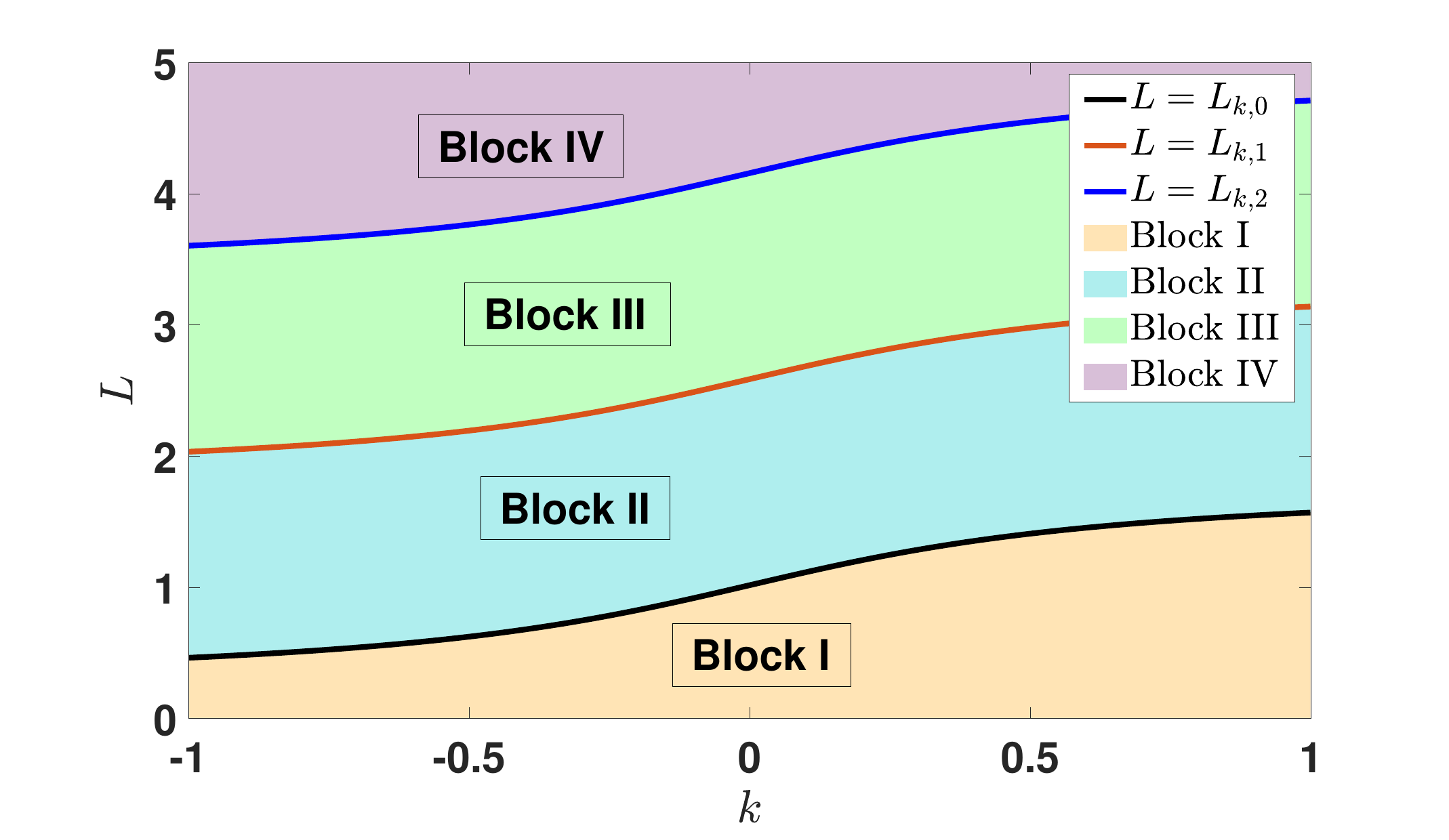}
			\label{figB111}
		}
		\centering
		\subfigure[]{
			\includegraphics[width=0.45\textwidth]{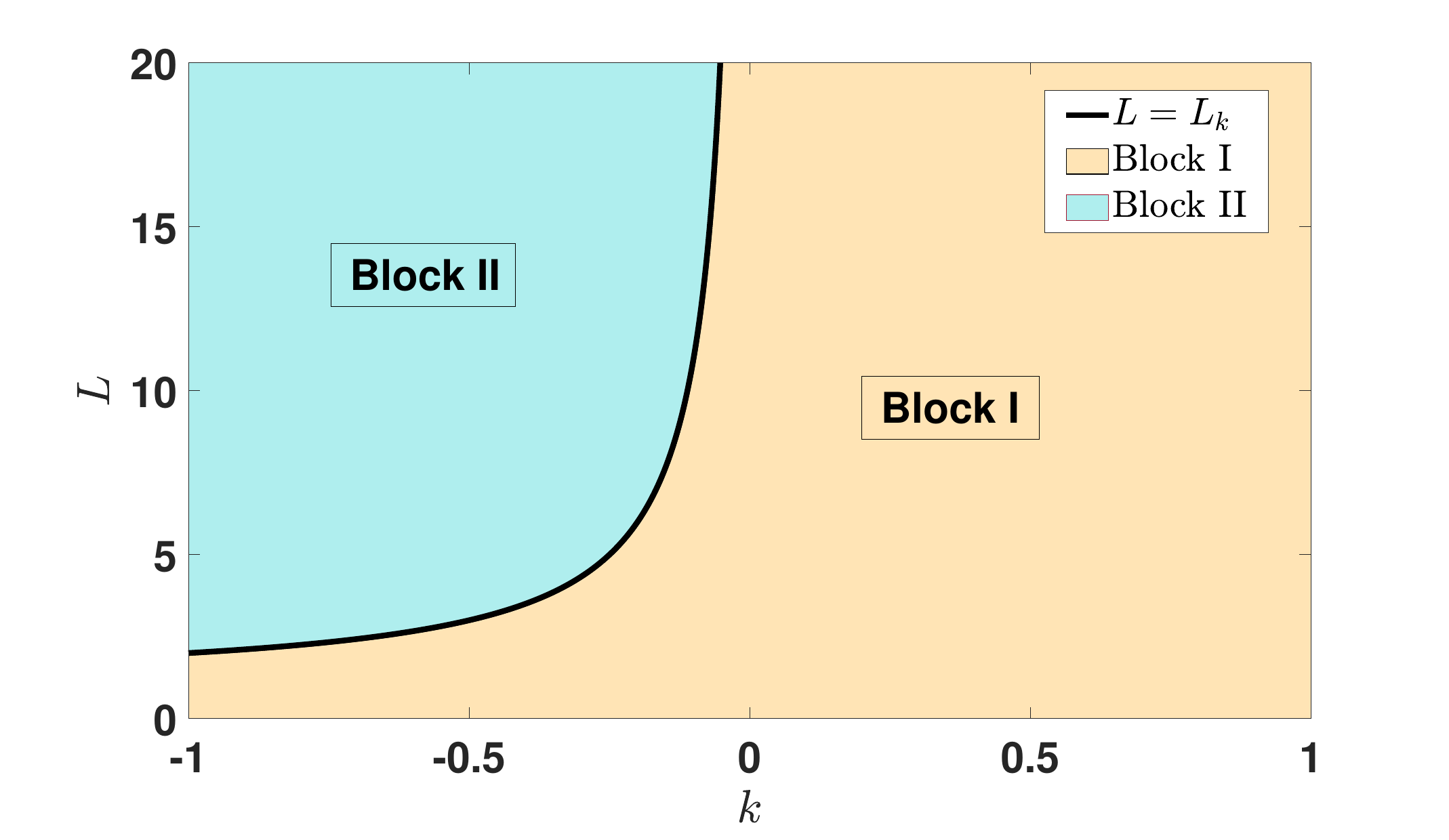}
			\label{figB222}
		}
		\centering
		\caption{$k-L$ plane is seperated by marginal curves determined by $\mathcal{A}_{a,b,\lambda}$ for different cases.}\label{figAA}
	\end{figure}
	The calculation of $\mathcal{A}_{a,b,\lambda}$ is provided in Appendix \ref{appendixa}. It is worth noting that $N_{a,b,\lambda}$ is a constant within each block as per Lemma \ref{lemma2}. Moreover, for $L=0$, we have $F_{a,b,\lambda,k,L}(\sigma)=k-1\ne0$ for $k\in(-1,1)$, which implies that $N_{a,b,\lambda}(k,0)=0$. As a result, $N_{a,b,\lambda}({\rm Block~I})=0$, and the corresponding system \eqref{system} \eqref{boun} with $(k,L)$ in Block I is exponentially stable.
\par	We can further demonstrate that if a point $(k,L)$ moves from one block to a block above it, then $N_{a,b,\lambda}$ increases by $1$. Therefore, for any point $(k,L)$ in a block other than Block I, the system \eqref{system} \eqref{boun} possesses at least one eigenvalue in $\overline{\mathbb{C}}_+$ and cannot be exponentially stable.
	
	\begin{lem}\label{lemma3}
		For $(k_0,L_0)\in\mathcal{A}_{a,b,\lambda}$, there exists $\varepsilon>0$, such that
		$$
		N_{a,b,\lambda}(k_0,L_0+\varepsilon)=N_{a,b,\lambda}(k_0,L_0-\varepsilon)+1.
		$$
	\end{lem}
	\begin{proof}[{\bf Proof of Lemma \ref{lemma3}}]
		From Appendix \ref{appendixa}, if $(k_0,L_0)\in\mathcal{A}_{a,b,\lambda}$, we have:
		$$F_{a,b,\lambda,k_0,L_0}(0)=0,$$
		with
		$$F_{a,b,\lambda,k,L}(\sigma)= (k-1)\cosh(\eta L)-\big{[}(k+1)\frac{\lambda+1}{2\lambda}\sigma+(\frac{kb}{\lambda}+a)\big{]}\frac{\sinh(\eta L)}{\eta}.$$
	\par	Denote $\sigma_0=0 $ and the corresponding $\eta$ is denoted by $\eta_0$, from Eq.\eqref{e11},
		\begin{equation}\label{e10}
			\eta_0\coth(\eta_0 L_0)=\frac{(\frac{kb}{\lambda}+a)}{(k-1)}.
		\end{equation}
\par		Denote
		$$
		H_{a,b,\lambda,k}(\sigma,L)\triangleq(k-1)\cosh(\eta L)-\big{[}(k+1)\frac{\lambda+1}{2\lambda}\sigma+(\frac{kb}{\lambda}+a)\big{]}\frac{\sinh(\eta L)}{\eta},
		$$
		with
		$$\eta^{2}=\frac{(\lambda+1)^{2}\sigma^{2}-4\lambda ab}{4\lambda^{2}},$$
		and
		$$
		\widetilde{H}_{a,b,\lambda,k}(\eta,\sigma,L)\triangleq(k-1)\cosh(\eta L)-\big{[}(k+1)\frac{\lambda+1}{2\lambda}\sigma+(\frac{kb}{\lambda}+a)\big{]}\frac{\sinh(\eta L)}{\eta}.
		$$
		\begin{itemize}
			\item If $ab\ne 0$, $\eta_0\ne 0$, at point $(\sigma_0,L_0)$, we have
			$$
			\frac{\partial\eta}{\partial\sigma}\bigg|_{(\sigma,L)=(0,L_0)}=\frac{(\lambda+1)^2}{4\lambda^2}\frac{\sigma}{\eta}=0.
			$$
		\par	The partial diffierential derivative of $H_{a,b,\lambda,k}(\sigma,L)$ with respect to $\sigma$ is
			$$
			\frac{\partial H_{a,b,\lambda,k}}{\partial \sigma}\bigg|_{(\sigma,L)=(0,L_0)}=\frac{\partial \widetilde{H}_{a,b,\lambda,k}}{\partial \eta}\cdot\frac{\partial\eta}{\partial\sigma}+\frac{\partial \widetilde{H}_{a,b,\lambda,k}}{\partial \sigma}=\frac{\partial \widetilde{H}_{a,b,\lambda,k}}{\partial \sigma}=-(k+1)\frac{\lambda+1}{2\lambda}\frac{\sinh(\eta_0 L_0)}{\eta_0}\neq 0.
			$$
		\par	Notice $\eta_0$ is a pure imaginary number ($ab>0$) or a real number ($ab<0$), from Eq.\eqref{e10}, the partial diffierential derivative of $H_{a,b,\lambda,k}(\sigma,L)$ with respect to $L$ is:
			$$
			\begin{aligned}
				\frac{\partial H_{a,b,\lambda,k}}{\partial L}\bigg|_{(\sigma,L)=(0,L_0)}=&\bigg[(k-1)\sinh(\eta_0 L_0)-(\frac{kb}{\lambda}+a)\frac{\cosh(\eta_0 L_0)}{\eta_0}\bigg]\eta_0\\
				=&\bigg[(k-1)\sinh(\eta_0 L_0)-(k-1)\frac{\cosh^2(\eta_0 L_0)}{\sinh(\eta_0 L_0)}\bigg]\eta_0\\
				=&-\frac{(k-1)\eta_0}{\sinh(\eta_0 L_0)}.
			\end{aligned}
			$$
	\par		Using the Implicit Function Theorem \ref{l1}, there exists an implicit function $\sigma(L)$ defined on $(L_0-\epsilon,L_0+\epsilon)$ for small sufficiently $\epsilon>0$ such that 
			$$
			H_{a,b,\lambda,k}(\sigma(L),L)=0, \sigma(L_0)=0,$$
			and $\sigma'(L)=-\bigg(\frac{\partial H_{a,b,\lambda,k}}{\partial L}\bigg)/\bigg(\frac{\partial H_{a,b,\lambda,k}}{\partial \sigma}\bigg)$. More precisely, 
			$$
			\sigma'(L_0)=\frac{-2\lambda(k-1)}{(\lambda+1)(k+1)}\cdot\frac{\eta_0^2}{\sinh^2(\eta_0 L_0)}>0.
			$$
	\par		The last inequality is obtained by $\frac{\eta_0}{\sinh(\eta_0L_0)}\in\mathbb{R}\backslash\{0\}$. Then we finish the proof with $ab\ne0$.
			\item If $ab=0$, we have the following characteristic equation	
			\begin{equation}\label{ee1}
				(k-1)\cosh(\eta L)-\big{(}(k+1)\frac{\lambda+1}{2\lambda}\sigma+(\frac{kb}{\lambda}+a)\big{)}\frac{\sinh(\eta L)}{\eta}=0,
			\end{equation}
			with $\eta=\frac{(\lambda+1)\sigma}{2\lambda}$.
	\par At point $(\sigma_0,L_0)$, we obtain
			$$
			(k-1)-(\frac{kb}{\lambda}+a)L_0=0.
			$$
	\par	Moreover,
			$$
			\frac{\partial\eta}{\partial\sigma}|_{(\sigma,L)=(0,L_0)}=\frac{\lambda+1}{2\lambda}.
			$$
	\par	Since $\cosh(\eta L)$ and $\frac{\sinh(\eta L)}{\eta}$ are analytic functions with respect to $\eta^2$, we obtain
			$$
			\frac{\partial \widetilde{H}_{a,b,\lambda,k}}{\partial \eta}|_{(\eta,\sigma,L)=(0,0,L_0)}=0.
			$$
	\par	Thus,
			$$
			\begin{aligned}
				&\frac{\partial H_{a,b,\lambda,k}}{\partial \sigma}|_{(\sigma,L)=(0,L_0)}=\frac{\partial \widetilde{H}_{a,b,\lambda,k}}{\partial \eta}\cdot\frac{\partial\eta}{\partial\sigma}+\frac{\partial \widetilde{H}_{a,b,\lambda,k}}{\partial \sigma}\\
				=&-(k+1)\frac{\lambda+1}{2\lambda}\frac{\sinh(\eta_0 L_0)}{\eta_0}=-(k+1)\frac{\lambda+1}{2\lambda}L_0\ne0.
			\end{aligned}
			$$
	\par	The partial diffierential derivative of $H_{a,b,\lambda,k}(\sigma,L)$ with respect to $L$ is
			$$
			\frac{\partial H_{a,b,\lambda,k}}{\partial L}|_{(\sigma,L)=(0,L_0)}=-(\frac{kb}{\lambda}+a).
			$$
	\par	Using the Implicit Function Theorem \ref{l1}, there exists an implicit function $\sigma(L)$ defined on $(L_0-\epsilon,L_0+\epsilon)$ for small sufficiently $\epsilon>0$  such that 
			$$
			H_{a,b,\lambda,k}(\sigma(L),L)=0, \sigma(L_0)=0,
			$$
			and $\sigma'(L)=-\bigg(\frac{\partial H_{a,b,\lambda,k}}{\partial L}\bigg)/\bigg(\frac{\partial H_{a,b,\lambda,k}}{\partial \sigma}\bigg)$. More precisely,
			$$
			\sigma'(L_0)=-\frac{2\lambda (k-1)}{(\lambda+1)(k+1)L_0^2}>0.
			$$
	\par		Then we finish the proof of $ab=0$. \end{itemize}
	\end{proof}
~\\
\noindent\textbf{Proof of Theorem \ref{main theorem}}
\par Lemma \ref{lemma3} gives us a way to determine $N_{a,b,\lambda}$ for each block. For example, for the case $ab>0$, the $k-L$ plane is seperated into infinitely blocks by marginal curves determined by $\mathcal{A}_{a,b,\lambda}$ (see Fig. \ref{figB11}). Thus, we know that $N_{a,b,\lambda}({\rm Block~I})=0,\ N_{a,b,\lambda}({\rm Block~II})=1,\ N_{a,b,\lambda}({\rm Block~III})=2,\  N_{a,b,\lambda}({\rm Block~IV})=3$ and so on. For the case $a>0,b=0$, there are no marginal curves (See Fig. \ref{figB22}). Thus, we know that for all the point $(k,L)\in(-1,1)\times [0,+\infty)$, we have $N_{a,b,\lambda}(k,L)=0$.
	\begin{figure}[H]
		\centering
		\subfigure[]{
			\includegraphics[width=0.45\textwidth]{B1.pdf}
			\label{figB11}
		}
		\centering
		\subfigure[]{
			\includegraphics[width=0.45\textwidth]{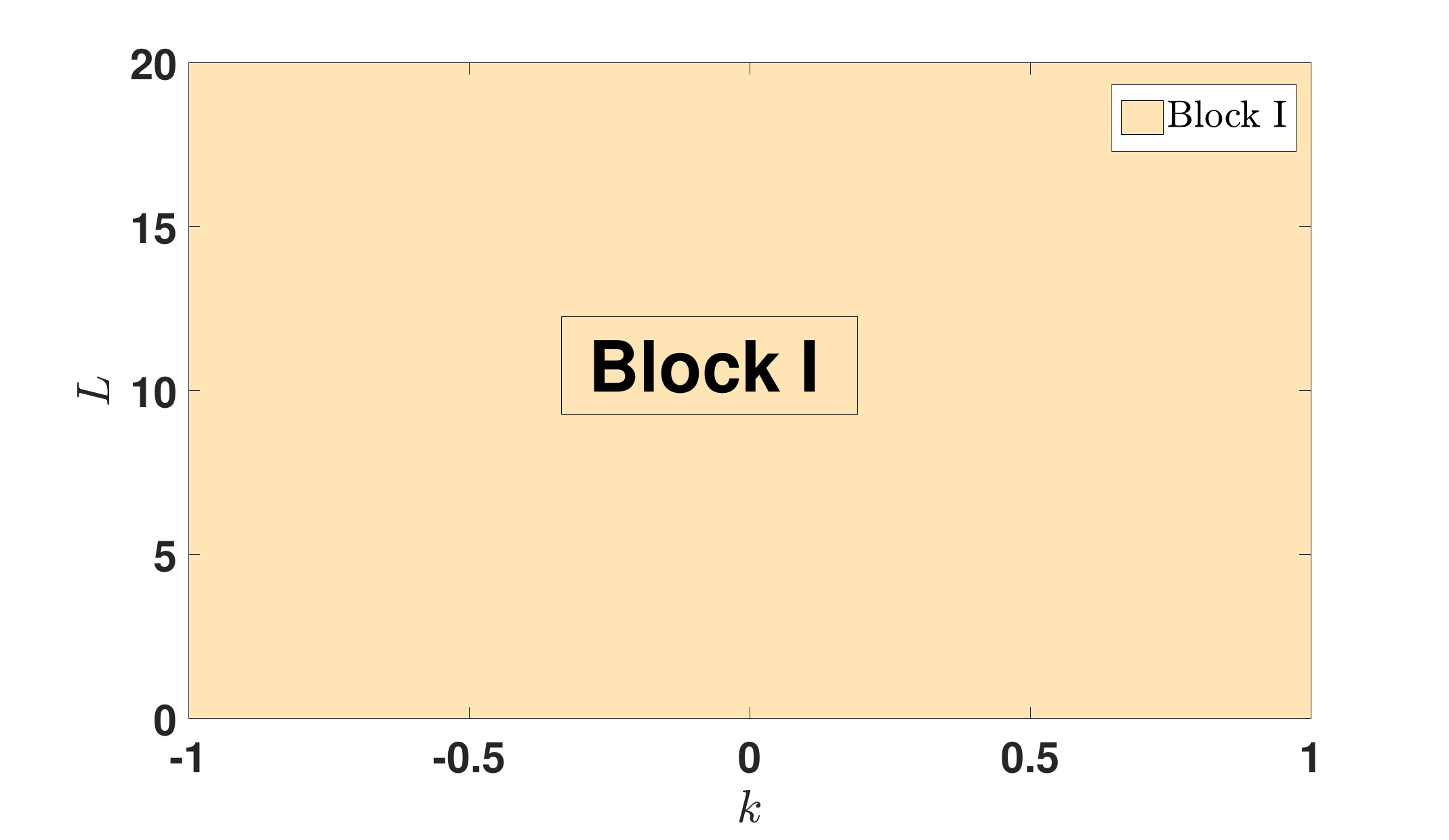}
			\label{figB22}
		}
		\centering
		\caption{$k-L$ plane is seperated by marginal curves determined by $\mathcal{A}_{a,b,\lambda}$ for different cases.}\label{figBB}
	\end{figure}
	Thus we know that the stability region is the block at the bottom which contains  $\{(k,0)|k\in(-1,1)\}$. Denote $D(a,b,\lambda)$ is the block that contains  $\{(k,0)|k\in(-1,1)\}$, we obtain
	$$
	L_c(a,b,\lambda)=\max\{L|(k,L)\in D(a,b,\lambda)\}.
	$$
\par	Moreover, if $\mathcal{A}_{a,b,\lambda}=\emptyset$, for any $L>0$, the corresponding system \eqref{system} \eqref{boun} with $(k,L)$ is exponentially stable with $L_{c}\triangleq+\infty$ that is defined in Theorem \ref{main theorem}.
	More precisely, from the calculation in Appendix \ref{appendixa}, we obtain
	\begin{equation}
		\begin{aligned}
			\displaystyle	
			L_c=\begin{cases}
				\sqrt{\frac{\lambda}{ab}}\pi,\quad &\text{if}\ a>0,b>0.\\
				\sqrt{\frac{\lambda}{ab}}\arccot(\frac{b-\lambda a}{2\sqrt{\lambda ab}}),\quad &\text{if}\ a<0, b<0.\\
				\sqrt{\frac{-\lambda}{ab}}\coth^{-1}(\frac{b-\lambda a}{2\sqrt{-\lambda ab}}),\quad &\text{if}\ -\lambda a>b>0.\\
				-\frac{2}{a},\quad &\text{if}\ b=0, a<0.\\
				+\infty,\quad &\text{if else.}
			\end{cases}
		\end{aligned}
	\end{equation}
	

	\section{Numerical Simulations}
	In this section we present some numerical simulations generated with MATLAB of upwind scheme with implicit methods for the system (\ref{system}) (\ref{boun}).  We adopt the finite difference method in both the time and the space domain, which can be written as follows. The grid size $N=100$ and the time step $\Delta t=2L/N$ are used.  
	$$
	\begin{cases}
		
		\frac{u_j^{n+1}-u_j^{n}}{\Delta t}+\frac{u_j^{n+1}-u_{j-1}^{n+1}}{\Delta x}+av_{j}^{n+1}=0,\quad &j=1,...,N;n=0,...,\frac{T}{\Delta t},\\
		\frac{v_j^{n+1}-v_j^{n}}{\Delta t}-\lambda\frac{v_{j+1}^{n+1}-v_{j}^{n+1}}{\Delta x}+bu_{j}^{n+1}=0,\quad& j=0,...,N-1;n=0,...,\frac{T}{\Delta t},\\
		u_0^{n}=kv_0^{n},\quad& n=0,...,\frac{T}{\Delta t},\\
		u_N^{n}=v_N^{n},\quad& n=0,...,\frac{T}{\Delta t}.
		
	\end{cases}
	$$
	\par	Here $u_j^{n}$ and $v_j^{n}$ provide an approximation of $y_1(x_j,t_n)$ and $y_2(x_j,t_n)$, respectively. The initial conditions are chosen as
	$$
	\begin{cases}
		y_{1}(0,x)=x+\sin^{2}x,\\
		y_{2}(0,x)=\frac{L+\sin^2L}{L+L^2}(x^{2}+x).
	\end{cases}$$
	\par	Energy is measured in the $L_2$-norm for
	$$
	E_t\triangleq \int_0^L [y_1(t,x)^2+y_2(t,x)^2]{\rm d}x.
	$$
	\par We choose four triple $(a,b,\lambda)=(1,1,1),(-1,-2,1),(0,1,1),(-1,0,1)$. For each triple value of $(a,b,\lambda)$, we numerically implement system \eqref{system} for the parameter $(k,L)\in (-1,1)\times(0,3)$. As shown in Fig. \ref{fig1}, the analytical creteria for stability region established in Section 3 can be confirmed by our numerical simulations. Specially, for $a=-1,b=-2,\lambda=1,L=0.9$, Fig. \ref{fig1b} shows that the system will be stabilized only when the coupling gain $k\in(-1,k_*)$ for $k_*\approx -0.23$. Numerical confirmation is shown in Fig. \ref{fig2}. For different values of coupling gain $k$, the energy converges or diverges exponentially with different rates.
	\begin{figure}[H]
		\centering
		\includegraphics[height=6cm]{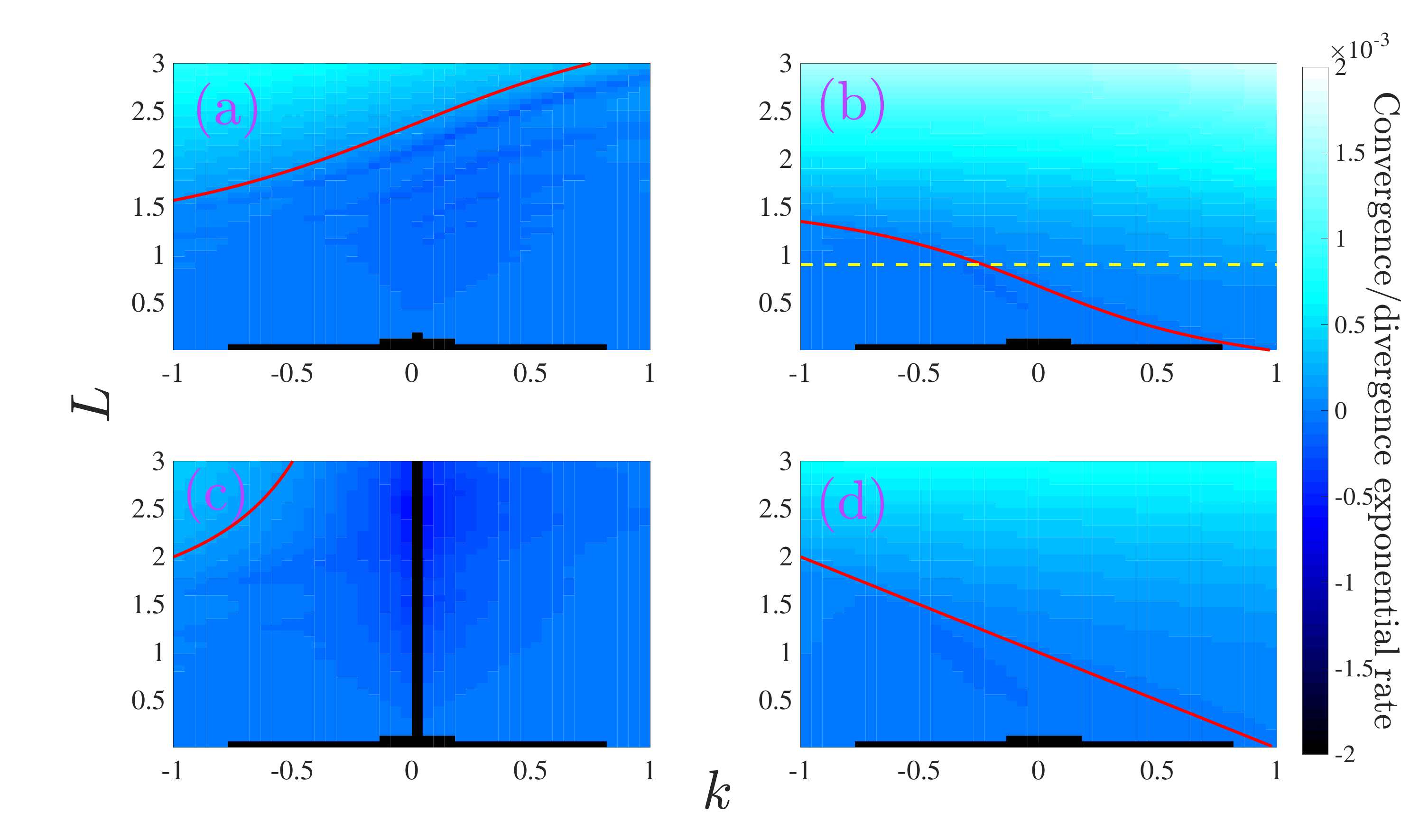}
		\subfigure{\label{fig1a}}
		\subfigure{\label{fig1b}}
		\subfigure{\label{fig1c}}
		\subfigure{\label{fig1d}}
		\caption{Red curves are depicted by the analytical results according to Appendix.\ref{appendixa}, below which are the stability region for $k,L$. The colors represent the exponential rates of the convergence or divergence of the  trajectory numerically generated by Eq.\eqref{system}. The parameters are (a) $a=1,b=1,\lambda=1$ (b) $a=-1,b=-2,\lambda=1$ (c) $a=0,b=1,\lambda=1$ (d) $a=-1,b=0,\lambda=1$. Yellow dashed line in (b) represents $L=0.9$, which will be investigated more clearly in Fig. \ref{fig2}.}\label{fig1}
	\end{figure}
	\begin{figure}[H]
		\centering
		\includegraphics[height=5cm]{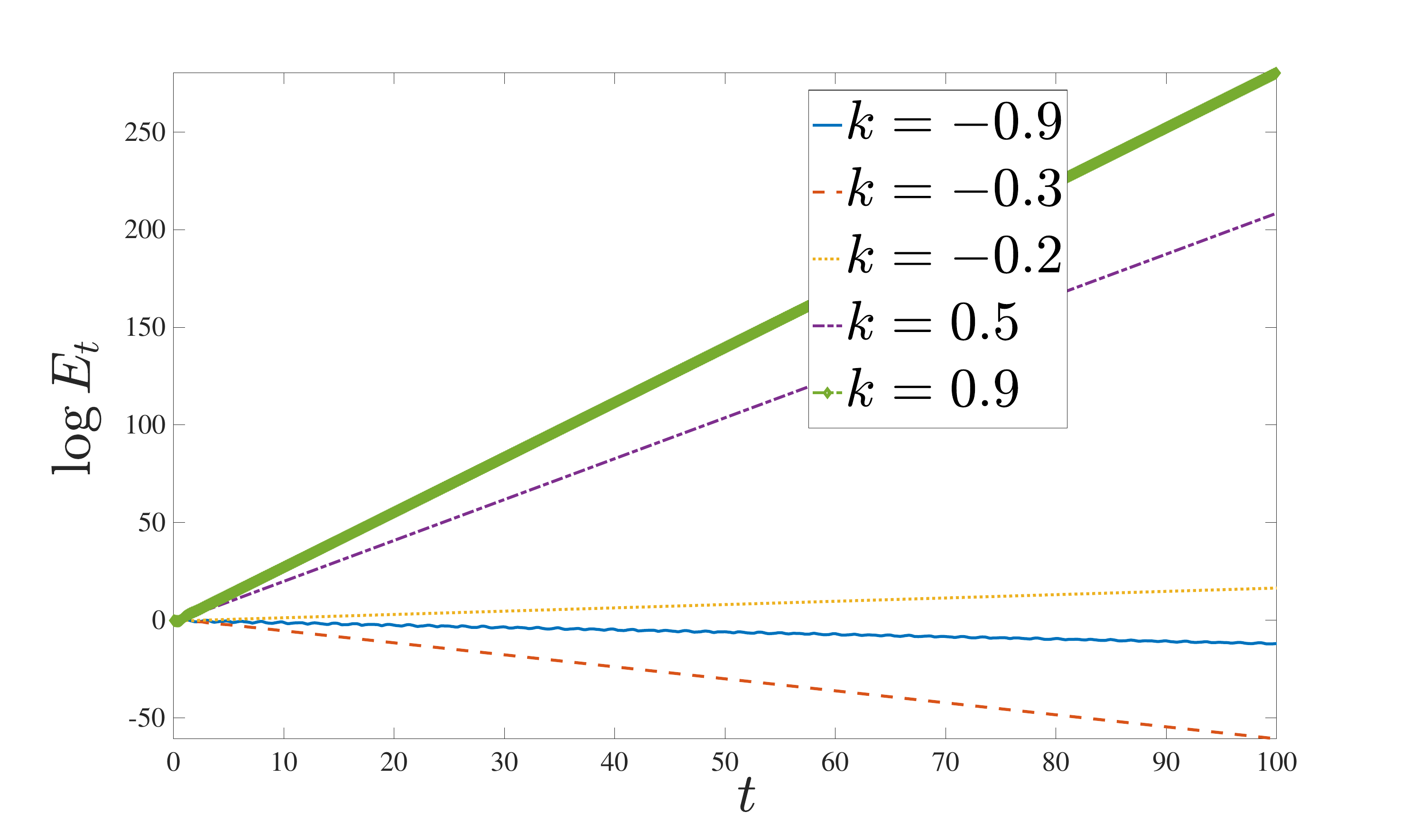}
		\caption{For different values of coupling gain $k$, the energy $E_t$ converges or diverges  exponentially with different rates. The parameters are $a=-1,b=-2,\lambda=1, L=0.9$.}\label{fig2}
	\end{figure}

	\section{Backstepping control}
	In this section, we want to use Backstepping method combined with the observer design to stabilize the system with the case that cannot be stabilized by the proportional feedback control.
	We first make a scaling of space variable $x \rightarrow L-x$, then the control could be on the right side and the boundary condition be on the left. Theorem \ref{main theorem} can apply to the following system:
	\begin{align}
		\begin{cases}\label{system3}
			\partial_{t}y_{1}-\partial_{x}y_{1}+ay_{2}=0,\quad&  (t,x)\in(0,+\infty)\times(0,L),\\
			\partial_{t}y_{2}+\lambda\partial_{x}y_{2}+by_{1}=0,\quad&  (t,x)\in(0,+\infty)\times(0,L),\\
			y_{2}(t,0)=y_{1}(t,0),\quad& t\in(0,+\infty),\\
			y_{1}(t,L)=U(t),\quad& t\in(0,+\infty).
		\end{cases}
	\end{align}
	\par The output is 
	\begin{equation}
		\label{boun2}Y(t)=y_{2}(t,L).
	\end{equation}
\par 	Applying the results in Anfinsen et al.\cite{An2017}, we design the following observer:
	\begin{align}
		\begin{cases}\label{system2}
			\partial_{t}\hat{y}_{1}-\partial_{x}\hat{y}_{1}+a\hat{y}_{2}=\Gamma_1(x)(Y(t)-\hat{y}_2(t,L)),\quad  (t,x)\in(0,+\infty)\times(0,L),\\
			\partial_{t}\hat{y}_{2}+\lambda\partial_{x}\hat{y}_{2}+b\hat{y}_{1}=\Gamma_2(x)(Y(t)-\hat{y}_2(t,L)),\quad  (t,x)\in(0,+\infty)\times(0,L),\\
			\hat{y}_{2}(t,0)=\hat{y}_{1}(t,0),\quad t\in(0,+\infty),\\
			\hat{y}_{1}(t,L)=U(t),\quad t\in(0,+\infty),
		\end{cases}
	\end{align}
	with $\Gamma_1(x), \Gamma_2(x)$ are injection gains to be designed.
\par	We have the following proposition:
	\begin{prop}\label{observer}
		Suppose the system (\ref{system3}) and the observer (\ref{system2}) with $\lambda,a,b\in\mathbb{R}$, $T_{opt1}\triangleq\frac{(\lambda+1)}{\lambda}L$. There exists suitable injection designs 
		$\Gamma_1(x), \Gamma_2(x)$ such that for all $t>T_{opt1}$, we have:
		\begin{align}
			\begin{cases}
				y_1(t,{x})=\hat{y}_{1}(t,x),\quad  (t,x)\in[T_{opt1},+\infty)\times[0,L], \\
				y_2(t,{x})=\hat{y}_{2}(t,x),\quad  (t,x)\in[T_{opt1},+\infty)\times[0,L].
			\end{cases}
		\end{align}	
	\end{prop}
	\begin{proof}[{\bf Proof of Proposition \ref{observer}}]
		 The state estimation errors $\widetilde{y}_1\triangleq y_1-\hat{y}_1,\widetilde{y}_2\triangleq y_2-\hat{y}_2$ satisfy the dynamics:
		\begin{align}\label{er1}
			\begin{cases}
				\partial_{t}\widetilde{y}_{1}-\partial_{x}\widetilde{y}_{1}+a\widetilde{y}_{2}=-\Gamma_1(x)\widetilde{y}_2(L),\quad&  (t,x)\in(0,+\infty)\times(0,L),\\
				\partial_{t}\widetilde{y}_{2}+\lambda\partial_{x}\widetilde{y}_{2}+b\widetilde{y}_{1}= 
				-\Gamma_2(x)\widetilde{y}_2(L),\quad&  (t,x)\in(0,+\infty)\times(0,L),\\
				\widetilde{y}_{2}(t,0)=\widetilde{y}_{1}(t,0),\quad& t\in(0,+\infty),\\
				\widetilde{y}_{1}(t,L)=0,\quad& t\in(0,+\infty).
			\end{cases}
		\end{align}
\par		Design the backstepping transformation:
		\begin{align}\label{ba1}
			\begin{cases}
				\widetilde{y}_{1}(t,x)=\alpha(t,x)+\int_{x}^{L}P^{1}(x,\xi)\beta(t,\xi){\rm d}\xi,\\
				\widetilde{y}_{2}(t,x)=\beta(t,x)+\int_{x}^{L}P^{2}(x,\xi)\beta(t,\xi){\rm d}\xi,
			\end{cases}
		\end{align}
		which makes system \eqref{er1} become the target system as follows.
		\begin{align}\label{ta1}
			\begin{cases}
				\partial_{t}\alpha-\partial_{x}\alpha+\int_x^L Q^1(x,\xi)\alpha(\xi){\rm d}\xi=0,\quad&  (t,x)\in(0,+\infty)\times(0,L),\\
				\partial_{t}\beta+\lambda\partial_{x}\beta+\int_x^L Q^2(x,\xi)\alpha(\xi){\rm d}\xi+b\alpha=0,\quad&  (t,x)\in(0,+\infty)\times(0,L),\\
				\alpha(t,0)=\beta(t,0),\quad& t\in(0,+\infty),\\
				\alpha(t,L)=0,\quad& t\in(0,+\infty).
			\end{cases}
		\end{align}
\par		The proof of existence of  $Q^1,Q^2,P^1,P^2$ can be found in \cite[Section B]{An2017}. 
		The solution of target system \eqref{ta1} will vanish in finite time for $t> T_{opt1}$. If we denote by $\Gamma_1(x)\triangleq\lambda P^{1}(x,L),\Gamma_2(x)\triangleq\lambda P^{2}(x,L)$, by the transformation \eqref{ba1}, we obtain that the solution $[\tilde{y_1},\tilde{y_2}]^{
			\rm T}$ of Eq.\eqref{er1} will vanish when $t>T_{opt1}$.
	\end{proof}
		It follows from Proposition \ref{observer} that the right side of Eq.\eqref{system2} vanishes when $t> T_{opt1}$, which makes it become a homogeneous linear hyperbolic system. By \cite[Section 2.2, Section 2.3]{Hu2019}, we know that there exists an invertible backstepping transformation 
		\begin{align}\label{ba2}
			\begin{cases}
				z(t,x)=\hat{y}_{1}(t,x)-\int_{0}^{x}[K^{11}(x,\xi)\hat{y}_{1}(t,\xi)+K^{12}(x,\xi)\hat{y}_{2}(t,\xi)]{\rm d}\xi,\\
				w(t,x)=\hat{y}_{2}(t,x)-\int_{0}^{x}[K^{21}(x,\xi)\hat{y}_{1}(t,\xi)+K^{22}(x,\xi)\hat{y}_{2}(t,\xi)]{\rm d}\xi,
			\end{cases}
		\end{align}
		with its inverse transformation
		\begin{align}\label{in2}
			\begin{cases}
				\hat{y}_{1}(t,x)=z(t,x)+\int_{0}^{x}[L^{11}(x,\xi)z(t,\xi)+L^{12}(x,\xi)w(t,\xi)]{\rm d}\xi,\\
				\hat{y}_{2}(t,x)=w(t,x)+\int_{0}^{x}[L^{21}(x,\xi)z(t,\xi)+L^{22}(x,\xi)w(t,\xi)]{\rm d}\xi,
			\end{cases}
		\end{align}
		which transforms system \eqref{system2} into the target system as follows.
		\begin{align}\label{ta2}
			\begin{cases}
			\partial_{t}z-\partial_{x}z=0,\quad& (t,x)\in(T_{opt1},+\infty)\times(0,L),\\
			\partial_{t}w+\lambda\partial_{x}w=g(x)w(t,0),\quad& (t,x)\in(T_{opt1},+\infty)\times(0,L),\\
				w(t,0)=z(t,0),\quad& t\in(T_{opt1},+\infty),\\
				z(t,L)=0,\quad& t\in(T_{opt1},+\infty).
			\end{cases}
		\end{align}
	\par	From transformation \eqref{in2} evaluated at $x=L$, noting that $z(t,L)\equiv 0$, we obtain the following feedback control laws for the system \eqref{system2}
		$$U(t)\triangleq \int_{0}^{L}[L^{11}(L,\xi)z(t,\xi)+L^{12}(L,\xi)w(t,\xi)]{\rm d}\xi.$$ This yields that $\hat{y}_1, \hat{y}_2$ vanishes in finite time $t\ge T_{opt}\triangleq T_{opt1}+T_{opt2}$ with $T_{opt1}=T_{opt2}=\dfrac{(\lambda+1)L}{\lambda}.$ Thus we get the main theorem in this section.
	
	\begin{thm}\label{back2}
		There exists a boundary feedback control law $U(t)$ for the system (\ref{system3}) with $\lambda,a,b\in\mathbb{R}$  such that, for every $Y_0\in L^2(0,L)$, the solution
		$Y\in C^0([0,+\infty);L^2(0,L))$ to (\ref{system3}) satisfies $Y(t)=0, \forall t\ge T_{opt},$ where $T_{opt}=\frac{2(\lambda+1)}{\lambda}L$.
	\end{thm}
	\begin{rem}
		Proposition \ref{observer} and Theorem \ref{back2} have designed boundary feedback controls that stablize the cases which cannot be stabilized by the proportional boundary feedback control mentioned in Theorem \ref{main theorem}.
	\end{rem}
	\begin{rem}
		For the following system:
		\begin{align}
			\begin{cases}
				\partial_{t}y_{1}-\lambda(x)\partial_{x}y_{1}+a(x)y_{2}=0,\quad& (t,x)\in(0,+\infty)\times(0,L),\\
				\partial_{t}y_{2}+\mu(x)\partial_{x}y_{2}+b(x)y_{1}=0,\quad& (t,x)\in(0,+\infty)\times(0,L),\\
				y_{2}(t,0)=y_{1}(t,0),\quad& t\in(0,+\infty),\\
				y_{1}(t,L)=U(t),\quad& t\in(0,+\infty).
			\end{cases}
		\end{align}
	with $\lambda(x),\mu(x)\in C^1([0,L]), a(x),b(x)\in C^0([0,1])$ are known functions that satisfied $\lambda(x),\mu(x)>0$ and $k\in\mathbb{R}$.
\par	The basic ideas of designing the observer in Proposition \ref{observer} and the backstepping control in Theorem \ref{back2} can be applied to this system. Therefore, the system could be stabilized to zero in finite time by a boundary control $U(t)$ depending on $y_2(\tau,L)(\tau\in(0,t))$.
	
	\end{rem}
	\section*{Acknowledgements}
	
	
	\appendix
	\section{Calculation on $\mathcal{A}_{a,b,\lambda}$}\label{appendixa}
	Recall that the characteristic equation \eqref{e11} is
	$$
	(k-1)\cosh(\eta L)-\big{[}(k+1)\frac{\lambda+1}{2\lambda}\sigma+(\frac{kb}{\lambda}+a)\big{]}\frac{\sinh(\eta L)}{\eta}=0.
	$$
	with 
	$$\eta^{2}=\frac{(\lambda+1)^{2}\sigma^{2}-4\lambda ab}{4\lambda^{2}}.$$
	and the definition of $\mathcal{A}_{a,b,\lambda}$ is
	$$
	\mathcal{A}_{a,b,\lambda}:=\big\{(k,L)||k|<1,L\ge 0,\ \text{there exists}\ \beta\in\R\ \text{such that}\ F_{a,b,\lambda,k,L}(\mathrm{i}\beta)=0\big\},
	$$
	\par First, if $L=0$, we have $F_{a,b,\lambda,k,L}(\sigma)=k-1\ne0$ for $k\in(-1,1)$, which means
	\begin{equation}\label{st}
		(k,0)\not\in\mathcal{A}_{a,b,\lambda}.
	\end{equation} 
	\par We discuss $\mathcal{A}_{a,b,\lambda}$ with $(k,L)\in(-1,1)\times(0,+\infty)$.
	We first observe that if $\sigma=\mathrm{i}\beta$, we have
	$$\eta^2\in\R,\ \cosh(\eta L)\in\R,\ \frac{\sinh(\eta L)}{\eta}\in\R.$$
	\par Thus, the imaginary part of Eq.\eqref{e11} is:
	\begin{equation}\label{a1}
		(k+1)\frac{\lambda+1}{2\lambda}\beta\frac{\sinh(\eta L)}{\eta}=0,
	\end{equation}
and the real part of Eq.\eqref{e11} is:
	\begin{equation}\label{a2}
		(k-1)\cosh(\eta L)-(\frac{kb}{\lambda}+a)\frac{\sinh(\eta L)}{\eta}=0.
	\end{equation}
	\par If $\frac{\sinh(\eta L)}{\eta}=0$ (motivated by Eq.\eqref{a1}), the left side of Eq.\eqref{a2} is $k-1\neq0$, thus $\frac{\sinh(\eta L)}{\eta}\ne0$.
	\par From Eq.\eqref{a1}, we obtain 
	$$\beta=0.$$
\par	Therefore,
	$$
	\mathcal{A}_{a,b,\lambda}=\big\{(k,L)||k|<1,L>0,F_{a,b,\lambda,k,L}(0)=0\big\},
	$$
and the characteristic equation can be simplified:
	\begin{equation}\label{a3}
		(k-1)\cosh(\eta L)-(\frac{kb}{\lambda}+a)\frac{\sinh(\eta L)}{\eta}=0,
	\end{equation}
	with 
	$$\eta^{2}=-\frac{ab}{\lambda}.$$
\par	We now divide into three cases:
	\begin{itemize}
		\item[Case I.] $ab=0$, the corresponding $\eta=0$. Eq.\eqref{a3} yields:
		$$
		(k-1)-(\frac{kb}{\lambda}+a)L=0.
		$$
\par	Define $L_k\triangleq\frac{k-1}{\frac{kb}{\lambda}+a}$, we obtain
	\begin{equation}\label{ab=0}
		\mathcal{A}_{a,b,\lambda}=\begin{cases}
			\big\{(k,L_k)|k\in(-1,0]\big\},\quad &\text{if}\ a=0,b>0.\\
			\big\{(k,L_k)|k\in[0,1)\big\},\quad &\text{if}\ a=0,b<0.\\
			\emptyset,\quad &\text{if}\ a\geq 0,b=0.\\
			\big\{(k,L_k)|k\in(-1,1)\big\},\quad &\text{if}\ a<0,b=0.
		\end{cases}
	\end{equation}
		\begin{figure}[H]
			\centering
			\subfigure[]{
				\includegraphics[width=0.45\textwidth]{A1.pdf}
				\label{figA1}
			}
			\centering
			\subfigure[]{
				\includegraphics[width=0.45\textwidth]{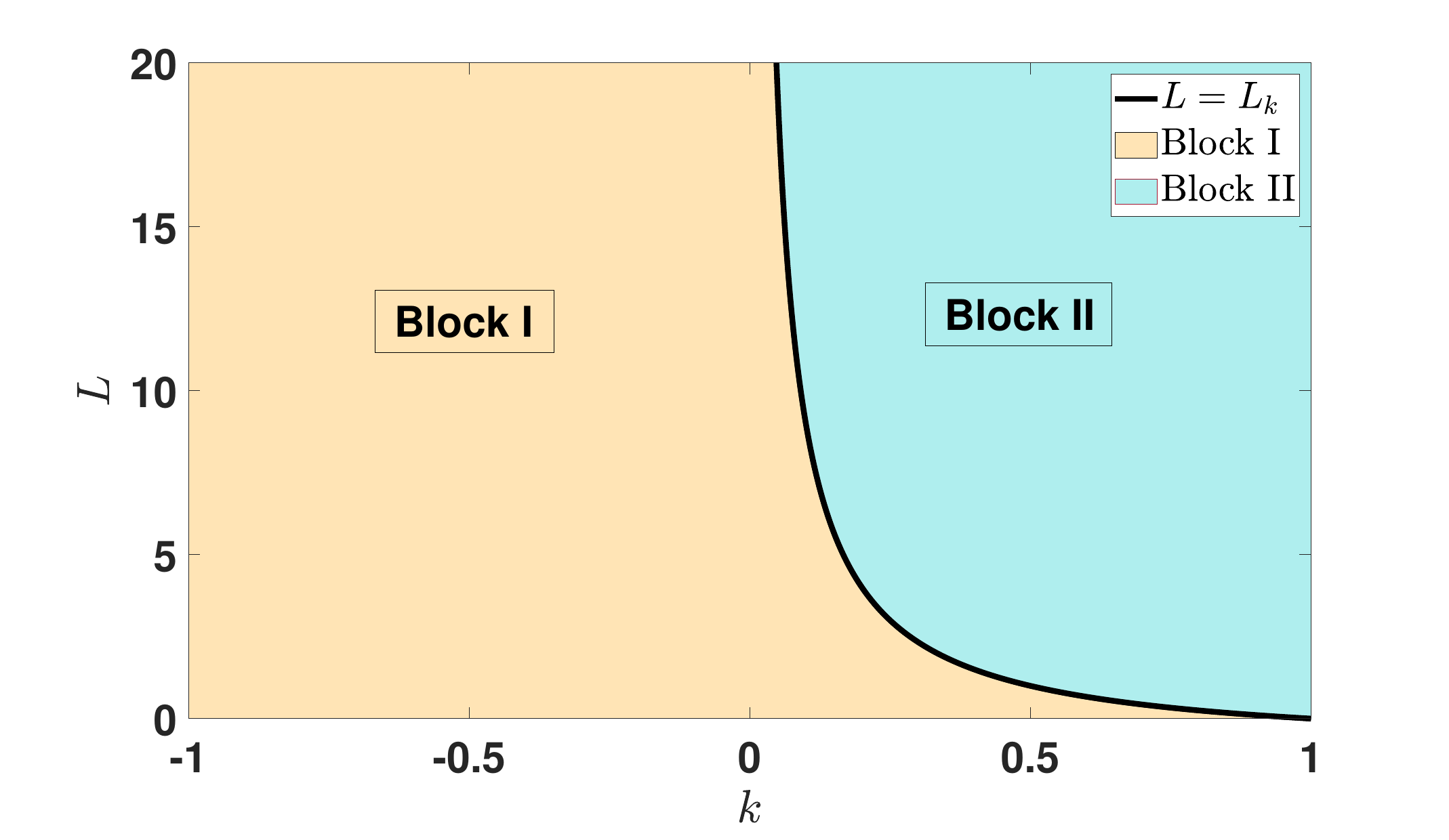}
				\label{figA2}
			}
			\subfigure[]{
				\includegraphics[width=0.45\textwidth]{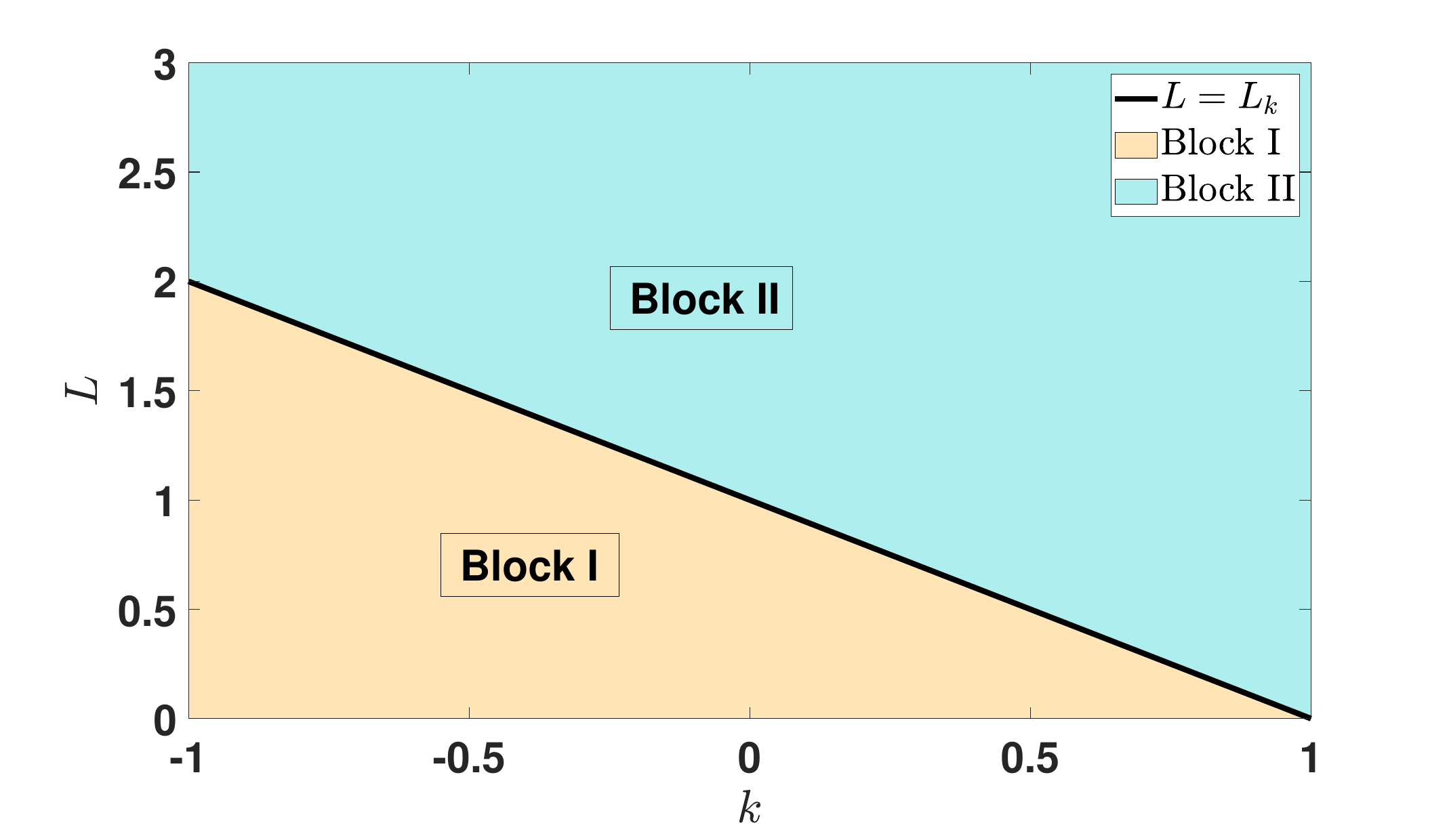}
				\label{figA3}}
			\subfigure[]{
				\includegraphics[width=0.45\textwidth]{C4.pdf}
				\label{figA4}}
			\centering
			\caption{$k-L$ plane is seperated by marginal curves determined by $\mathcal{A}_{a,b,\lambda}$ for the case $ab=0$. Black marginal  curves are determined by Eq.\eqref{ab=0}. The parameters are (a) $a=0,b=1,\lambda=1$,(b) $a=0,b=-1,\lambda=1$, (c) $a=-1,b=0,\lambda=1$, (d) $a=1,b=0,\lambda=1$. }\label{figA}
		\end{figure}
		That is
	\begin{equation}
		L_c=\begin{cases}
			+\infty,\quad &\text{if}\ a=0.\\
			+\infty,\quad &\text{if}\ a\geq 0,b=0.\\
			L_{-1}=-\frac{2}{a},\quad &\text{if}\ a<0,b=0.
		\end{cases}
	\end{equation}
		\item[Case II.] $ab>0$, Eq.\eqref{a3} yields:
		$$
		(k-1)\cos(\sqrt{\frac{ab}{\lambda}})-(\frac{kb}{\lambda}+a)\frac{\sin(\sqrt{\frac{ab}{\lambda}})}{\sqrt{\frac{ab}{\lambda}}}=0.
		$$
		It can be written as
		$$
		\cot(\sqrt{\frac{ab}{\lambda}}L)=\frac{\frac{kb}{\lambda}+a}{(k-1)\sqrt{\frac{ab}{\lambda}}}.
		$$
		We define
		\begin{equation}\label{ab>0}
			L_{k,n}\triangleq\sqrt{\frac{\lambda}{ab}}\bigg(\arccot\big(\frac{\frac{kb}{\lambda}+a}{(k-1)\sqrt{\frac{ab}{\lambda}}}\big)+n\pi\bigg)\ (n\in\N).
		\end{equation}
		Then we obtain
		$$
		\mathcal{A}_{a,b,\lambda}=\big\{(k,L_{k,n})|k\in(-1,1)\big\}.
		$$
		
		\begin{figure}[H]
			\centering
			\subfigure[]{
				\includegraphics[width=0.45\textwidth]{B1.pdf}
				\label{figB1}
			}
			\centering
			\subfigure[]{
				\includegraphics[width=0.45\textwidth]{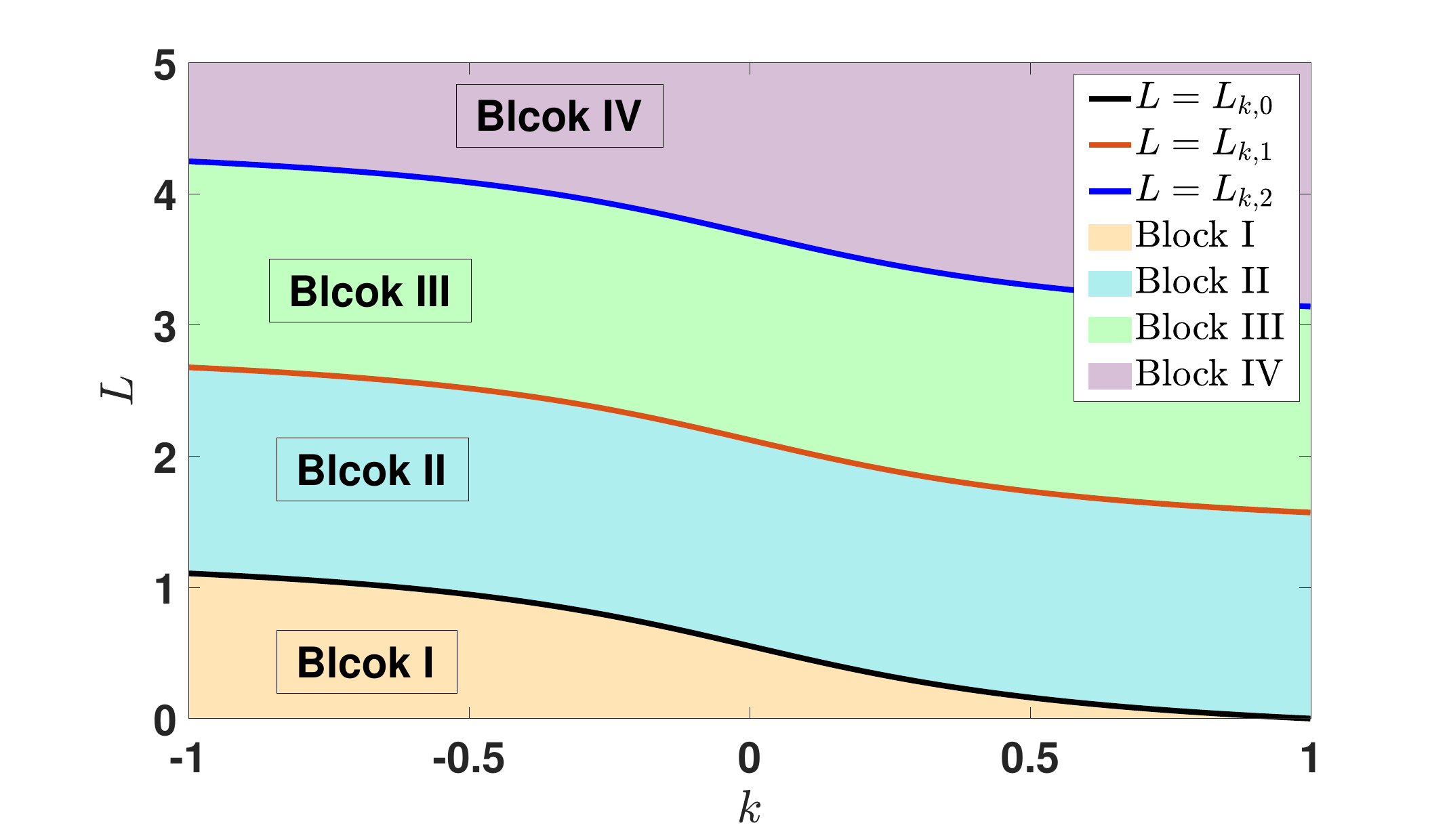}
				\label{figB2}
			}
			\centering
			\caption{$k-L$ plane is seperated by marginal curves determined by $\mathcal{A}_{a,b,\lambda}$ for the case $ab>0$. Black marginal  curves are determined by Eq.\eqref{ab>0}. The parameters are (a) $a=-1,b=-4,\lambda=1$,(b) $a=1,b=4,\lambda=1$.}\label{figB}
		\end{figure}
		That is, 
		\begin{equation}
			L_c=\begin{cases}
				L_{1,0}=\sqrt{\frac{\lambda}{ab}}\pi,\quad &\text{if}\ a,b>0.\\
				L_{-1,0}=\sqrt{\frac{\lambda}{ab}}\arccot(\frac{b-\lambda a}{2\sqrt{\lambda ab}}),\quad &\text{if}\ a,b<0.
			\end{cases}
		\end{equation}
		\item[Case III.] $ab<0$, Eq.\eqref{a3} yields:
		$$\coth(\sqrt{-\frac{ab}{\lambda}}L)=\frac{\frac{kb}{\lambda}+a}{(k-1)\sqrt{-\frac{ab}{\lambda}}}.$$
		Only if $\frac{\frac{kb}{\lambda}+a}{(k-1)\sqrt{-\frac{ab}{\lambda}}}>1$ the marginal curves exists, which can be defined by
		\begin{equation}\label{ab<0}
			L_k\triangleq\sqrt{-\frac{\lambda}{ab}}\coth^{-1}(\frac{\frac{kb}{\lambda}+a}{(k-1)\sqrt{-\frac{ab}{\lambda}}})>0.
		\end{equation}
		More precisely, let
		$$h(k)\triangleq\frac{\frac{kb}{\lambda}+a}{(k-1)\sqrt{-\frac{ab}{\lambda}}},$$
		which is continuous in $(k,L)\in(-1,1)\times(0,+\infty)$. Denote by
		$$
		k_1\triangleq\frac{-\lambda a-\sqrt{-\lambda ab}}{b-\sqrt{-\lambda ab}}=\sqrt{\dfrac{-\lambda a}{b}}{\rm Sgn}(a),
		$$
		with $h(k_1)=1$.
		\begin{itemize}
			\item[$\bullet$] $-\lambda a\geq b>0$, $h'(k)>0$. The range of $h(k)$ is $(h(-1),+\infty)$ with $h(-1)>1$.
			$$
			\mathcal{A}_{a,b,\lambda}=\big\{(k,L_k)|k\in(-1,1)\big\}.
			$$
			\item[$\bullet$] $0>-\lambda a>b$, $h'(k)>0$. The range of $h(k)$ is $(h(-1),+\infty)$ with $h(-1)<-1$.
			$$
			\mathcal{A}_{a,b,\lambda}=\big\{(k,L_k)|k\in(k_1,1)\big\}.
			$$
			\item[$\bullet$] $b>-\lambda a>0$, $h'(k)<0$. The range of $h(k)$ is $(-\infty,h(-1))$ with $h(-1)>1$.
			$$
			\mathcal{A}_{a,b,\lambda}=\big\{(k,L_k)|k\in(-1,k_1)\big\}.
			$$
			\item[$\bullet$] $0>b\geq-\lambda a$, $h'(k)<0$. The range of $h(k)$ is $(-\infty,h(-1))$ with $h(-1)<-1$. We obtain
			$L_k<0$ for $k\in(-1,1)$.
			$$
			\mathcal{A}_{a,b,\lambda}=\emptyset.
			$$
		\end{itemize}
		\begin{figure}[H]
			\centering
			\subfigure[]{
				\includegraphics[width=0.45\textwidth]{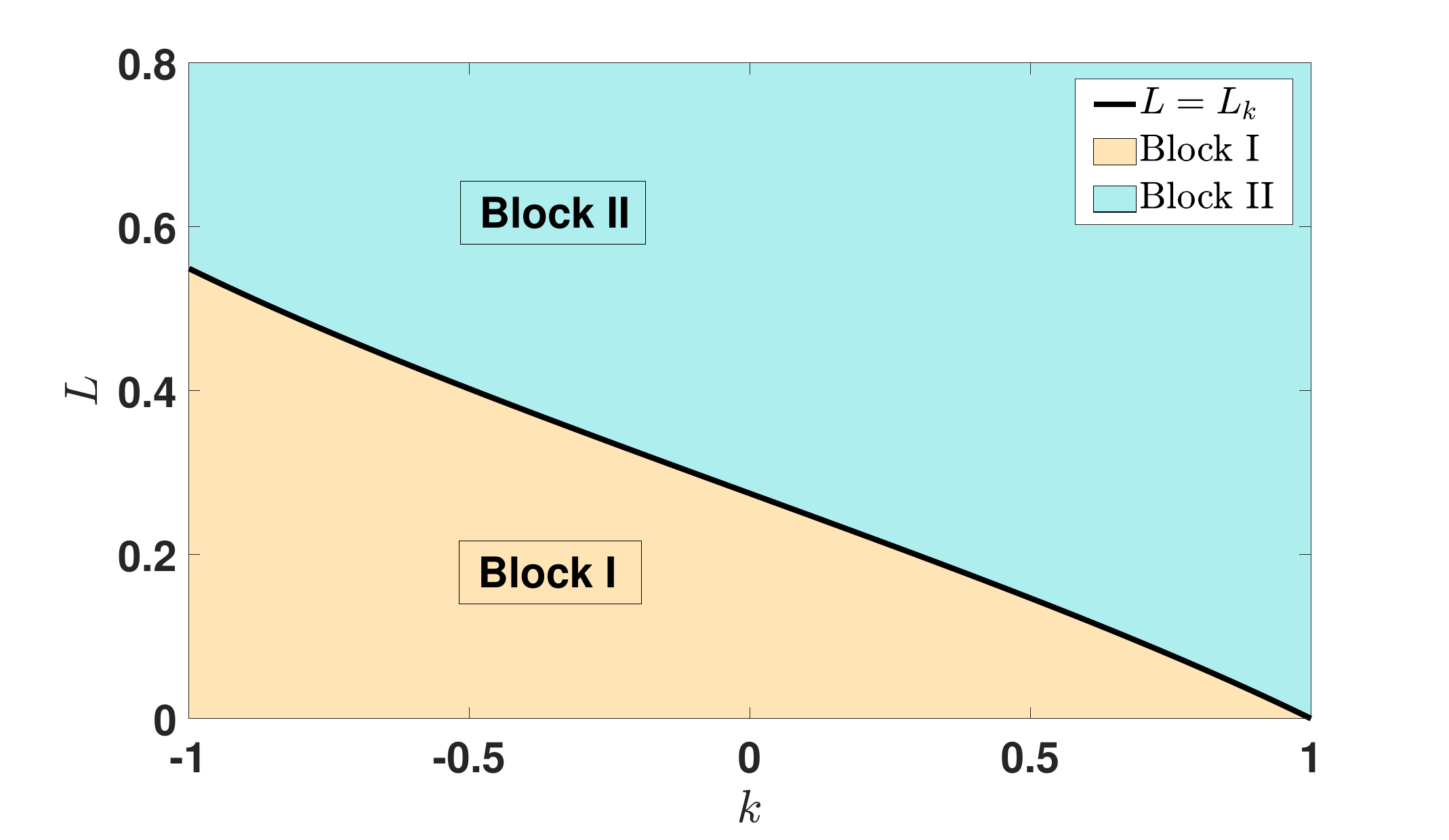}
				\label{figC1}
			}
			\centering
			\subfigure[]{
				\includegraphics[width=0.45\textwidth]{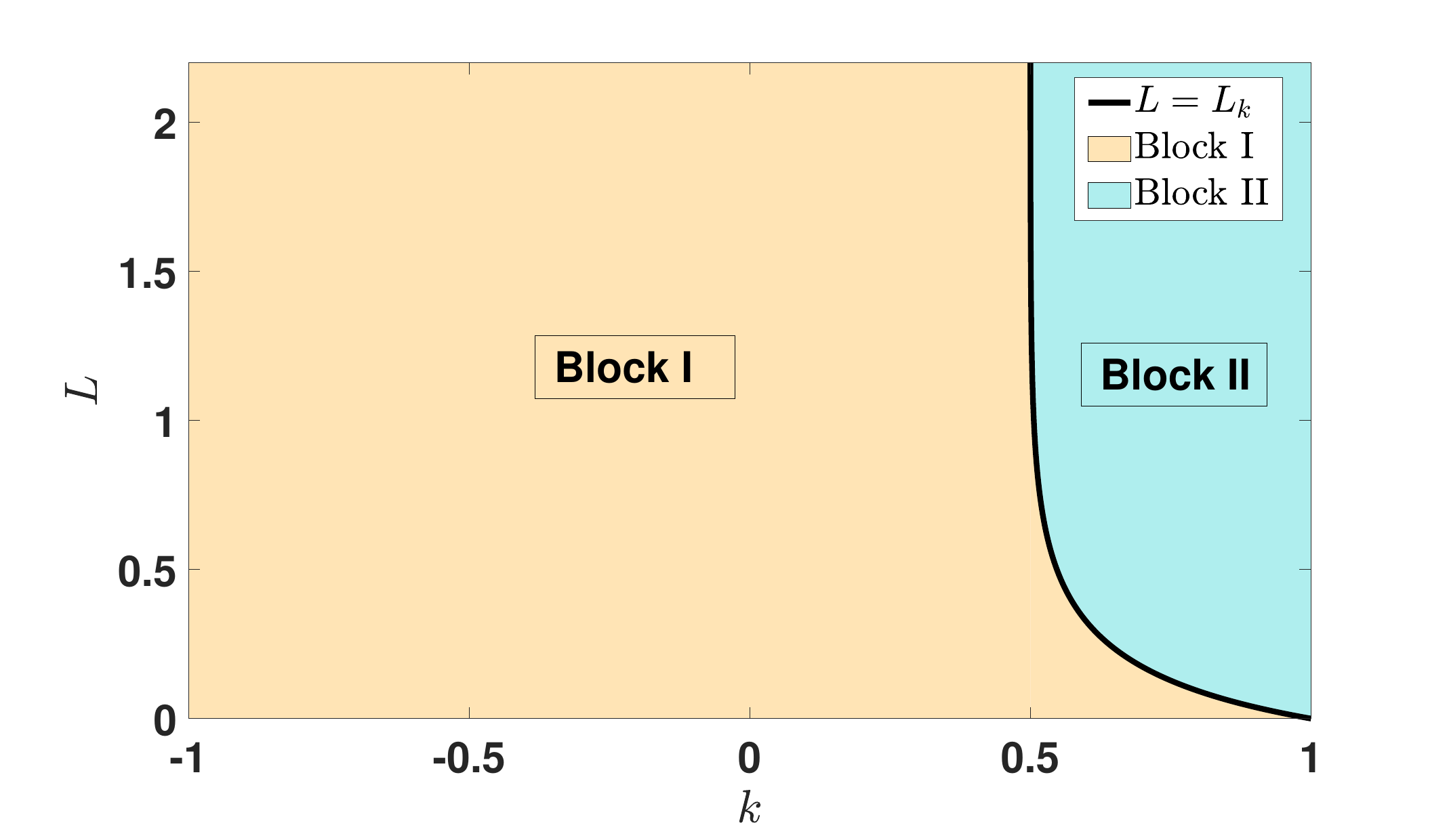}
				\label{figC2}
			}
			\subfigure[]{
				\includegraphics[width=0.45\textwidth]{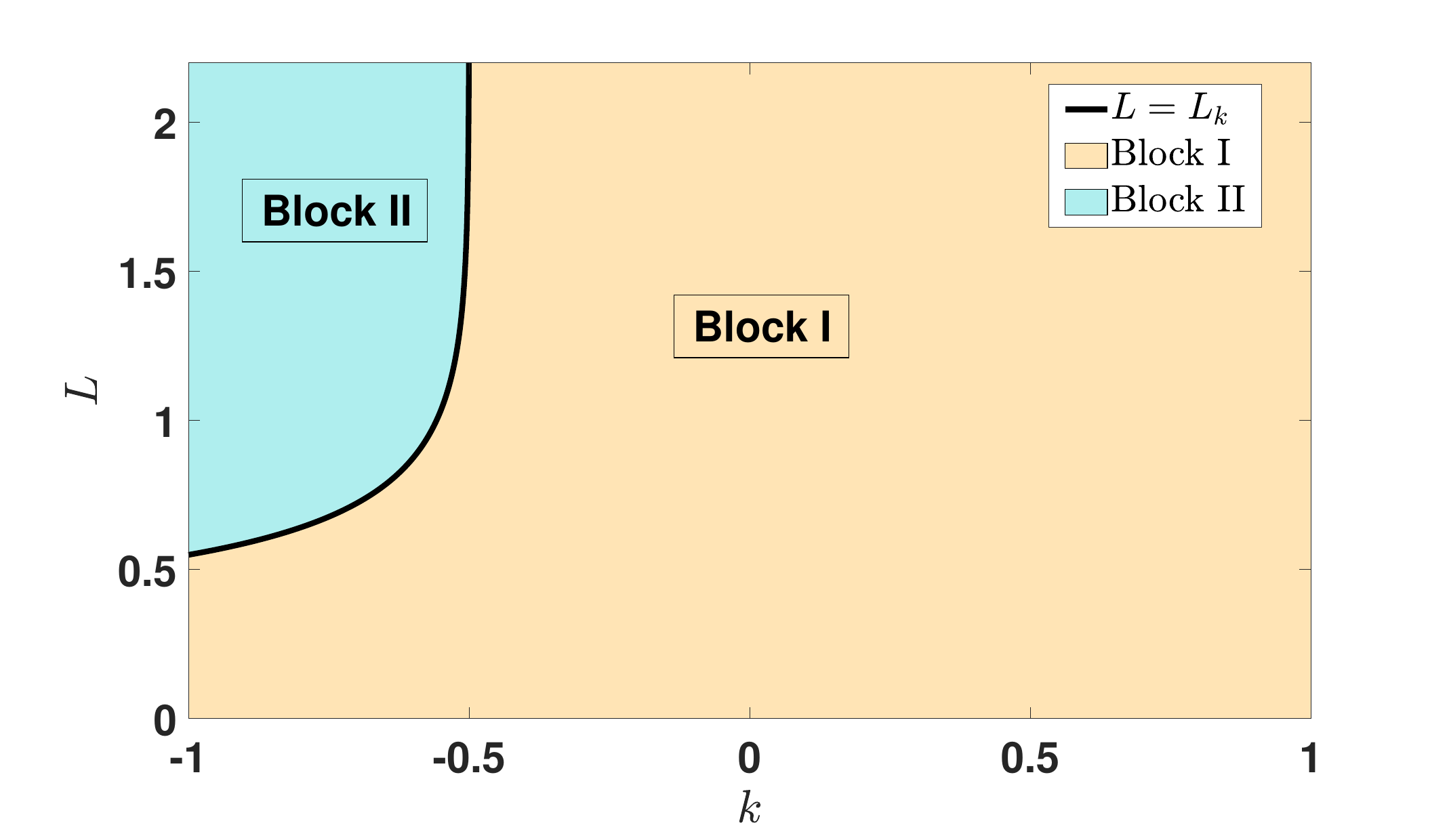}
				\label{figC3}}
			\subfigure[]{
				\includegraphics[width=0.45\textwidth]{C4.pdf}
				\label{figC4}}
			\centering
			\caption{$k-L$ plane is seperated into several blocks by marginal curves determined by $\mathcal{A}_{a,b,\lambda}$ for the case $ab<0$. Black marginal  curves are determined by Eq.\eqref{ab<0}. The parameters are (a) $a=-4,b=1,\lambda=1$,(b) $a=1,b=-4,\lambda=1$, (c) $a=-1,b=4,\lambda=1$, (d) $a=4,b=-1,\lambda=1$.}\label{figC}
		\end{figure}
		
		
		It is not difficult to obtain while $ab<0$:
		$$L_c=\begin{cases}
			L_{-1}=\sqrt{-\frac{\lambda}{ab}}\coth^{-1}(\frac{b-\lambda a}{2\sqrt{-\lambda ab}}),\quad &\text{if}\ -\lambda a>b>0. \\
			+\infty,\quad&\text{else}.
		\end{cases}$$
	\end{itemize}
	We finish the discuss on the set $\mathcal{A}_{a,b,\lambda}$.
	
	\bibliographystyle{plain}
	\bibliography{20230504}
\end{document}